\documentclass[article]{amsart}
\usepackage{cases}
\usepackage{caption}
\usepackage{amsmath, amsfonts, pifont, amssymb}
\usepackage{verbatim}
\usepackage[bookmarks=true]{hyperref}
\usepackage{tikz}
\usetikzlibrary{positioning}
\usetikzlibrary{calc}
\usetikzlibrary{snakes,arrows}
\usetikzlibrary{quotes,angles}
\usepackage{geometry}
\usepackage{color}
\newtheorem{theorem}{Theorem}[section]

\newtheorem{lemma}[theorem]{Lemma}
\newtheorem{proposition}[theorem]{Proposition}
\newtheorem{corollary}[theorem]{Corollary}

\newcommand{\R}{\mathbb{R}}

\newcommand{\C}{\mathbb{C}}

\newcommand{\beq}{\begin{equation}}
	\newcommand{\eeq}{\end{equation}}
\newcommand{\beqq}{\begin{equation*}}
	\newcommand{\eeqq}{\end{equation*}}

\theoremstyle{definition}
\newtheorem{definition}[theorem]{Definition}

\theoremstyle{remark}
\newtheorem{remark}[theorem]{Remark}
\newtheorem{conjecture}{Conjecture}[section]
\numberwithin{equation}{section}
\newcommand{\T}{\mathbb{T}}
\newcommand{\Z}{\mathbb{Z}}
\newcommand{\abs}[1]{\lvert#1\rvert}
\newcommand{\norm}[1]{\|#1\|}

\DeclareDocumentCommand{\abs}{s m}{
	\operatorname{}
	\IfBooleanTF{#1}{#2}{\left|#2\right|}}

\DeclareDocumentCommand{\norm}{s m}{
	\operatorname{}
	\IfBooleanTF{#1}{#2} {\left\| #2\right\|}}

\DeclareDocumentCommand{\inner}{s m}{
	\operatorname{}
	\IfBooleanTF{#1}{#2}{\left \langle#2\right \rangle}}

\DeclareDocumentCommand{\parenthese}{s m}{
	\operatorname{}
	\IfBooleanTF{#1}{#2}{\left(#2\right)}}

\DeclareDocumentCommand{\square}{s m}{
	\operatorname{}
	\IfBooleanTF{#1} {#2}{\left[#2\right]}}

\DeclareDocumentCommand{\bracket}{s m}{
	\operatorname{}
	\IfBooleanTF{#1}{#2}{\left\{#2\right\}}}

\begin{document}
	\author{Yilin Song}%
	\address{Yilin Song
		\newline \indent The Graduate School of China Academy of Engineering Physics,
		Beijing 100088,\ P. R. China}
	\email{songyilin21@gscaep.ac.cn}
	\author{Ruixiao Zhang}%
	\address{Ruixiao Zhang,
		\newline \indent The Graduate School of China Academy of Engineering Physics,
		Beijing 100088,\ P. R. China}
	\email{zhangruixiao21@gscaep.ac.cn}
	
\title[Nonlinear Schr\"{o}dinger equations]{ \bf Global well-posedness for the defocusing cubic nonlinear Schr\"odinger equation on $\Bbb T^3$  }
\begin{abstract}
	In this article, we investigate the global well-posedness for the defocusing, cubic nonlinear Schr\"{o}dinger equation posed on $\T^3$ with intial data lying in its critical space $H^\frac{1}{2}(\T^3)$. By  establishing the linear profile decomposition,  and applied this to the concentration-compactness/rigidity argument, we prove that if the solution remains bounded in the critical Sobolev space throughout the maximal lifespan, i.e.  $u\in L_t^\infty{H}^\frac{1}{2}(I\times\T^3)$, then $u$ is global.    
\end{abstract}
\maketitle
\begin{center}
	\begin{minipage}{100mm}
		{ \small {{\bf Key Words:}  Nonlinear Schr\"{o}dinger equation, critical norm, concentration-compactness, compact manifold.}
			{}
		}\\
		{ \small {\bf AMS Classification:}
			{35Q55, 35R01, 37K06, 37L50}
		}
	\end{minipage}
\end{center}
\section{Introduction}
In this article, we study the global well-posedness for the defocusing cubic nonlinear Schr\"odinger equation posed on three-dimensional tori, which is given by
\begin{align}\label{cubic-intro}
	\begin{cases}
		i\partial_tu+\Delta u=|u|^2u, & (t,x)\in\R\times\T_\theta^3, \\
		u(0,x)=u_0,
	\end{cases}
\end{align}
where $u:\R\times\T^3\to\Bbb C$ is the unknown function and $\T_\theta^3=\R^3/(\theta_1\Bbb Z\times\theta_2\Bbb{Z}\times\theta_3\Bbb Z)$ with $\theta=(\theta_1,\theta_2,\theta_3)$ and $\theta_i\in\R_+$. The tori $T_\theta^3$ is called  irrational tori if there exists at least one $\theta_i$ is not rational number. Otherwise, we call it rational tori. Since the analysis will not depends on the rationality of the tori, we abbreviate it to $\T^3$.
\subsection{The nonlinear Schrodinger equation at critical regularity}
In general, the defocusing nonlinear Schrodinger equation with power-type nonlinearity is given by
\begin{align}\label{general-NLS}
	\begin{cases}
		i\partial_tu+\Delta_M u=|u|^{\alpha-1}u, & (t,x)\in\R\times M, \\
		u(0,x)=u_0,
	\end{cases}
\end{align}
where $M=\R^d$ or $M=\T^d$.
The solution to \eqref{general-NLS} satisfies the following two conservation laws:
\begin{align*}
	\mbox{Mass}:& M(u)=\int_{M}|u|^2dx=M(u_0), \\
	\mbox{Energy}:& E(u)=\frac{1}{2}\int_{M}|\nabla u|^2dx+\frac{1}{\alpha+1}\int_{M}|u|^{\alpha+1}dx.
\end{align*}

When $M=\R^d$, the equation is invariant under the following scaling transformation,
\begin{align*}
	u_\lambda(t,x)=\lambda^{-\frac{2}{\alpha-1}}u(\lambda^{-2}t,\lambda^{-1}x).
\end{align*}
Moreover, we have
\begin{align*}
	\|u_\lambda(0)\|_{\dot{H}^{s_c}(\R^d)}=\|u(0)\|_{\dot{H}^{s_c}(\R^d)}\Longleftrightarrow s_c=\frac{d}{2}-\frac{2}{\alpha-1}.
\end{align*}
When $s=s_c$, we call the equation is at critical regularity.
Associated to the two conservation laws, when $s_c\in(0,1)$, we call (\ref{general-NLS}) as inter-critical; when $s_c>1$, we call it energy-supercritical. When $s_c=0$ or $s_c=1$, we call it mass-critical and energy-critical. It is worth to note that \eqref{cubic-intro} is $\dot{H}^\frac{1}{2}$-critical. When $M=\T^d$, the geometry breaks down the scaling symmetry. However, we still use the $s_c$ as the classification of the critical regularity.

For the Euclidean spaces, i.e. $M=\R^d$, the local well-posedness follows from the fixed-point theorem and the following Strichartz estimates
\begin{align*}
	\big\|e^{it\Delta}f\big\|_{L_t^qL_x^r(\R\times\R^d)}\lesssim\|f\|_{L_x^2(\R^d)},
\end{align*} 
where $(q,r)$ satisfies
\begin{align*}
	\frac{2}{q}=d(\frac{1}{2}-\frac{1}{r}),\quad 2\leq q,r\leq\infty,\quad (q,r,d)\neq(2,\infty,2).
\end{align*}
For the case of $s_c=1$, Bourgain\cite{Bourgain1999} first developed the induction on energy method and combined with the localized Morawetz estimates to prove the global well-posedness and scattering in dimension three.  Colliander-Keel-Staffilani-Takaoka-Tao\cite{Colliander2008} exploited the interaction Morawetz estimates and established the  more complicated energy-induction argument to show the scattering for non-radial data in $3D$. Later, Ryckman-Visan\cite{RyckmanVisan2007} and Visan\cite{Visan2007} proved the scattering on four-dimension and higher-dimensions $d\geq5$ respectively.

For the case of $s_c=0$, Killip-Tao-Visan\cite{KillipTaoVisan} and Tao-Visan-Zhang\cite{TaoVisanZhang2007} proved the global well-posedness and scattering for the radial initial data with $d\geq2$. Dodson\cite{Dodson2012}
developed the long-time Strichartz estimates and applied this to establish the frequency-localized Morawetz estimates to prove the global well-posedness and scattering for general initial data with $d\geq3$. For the lower diemnsions $d=1,2$, since the endpoint Strichartz estimates fail, he further introduced the atomic spaces and the bilinear estimates associated with the minimal blow-up solutions to show the scattering in \cite{Dodson2016a} and \cite{Dodson2016b}. For the related mass critical problem, we refer to \cite{YZ, YZ2}.   

For the non-conserved critical regularity $s_c\neq0,1$, the global well-posedness is only known for the small data. Due to the lack of conserved quantities, there is no uniformly control in the growth of Sobolev norm throught the time direction. Therefore, the unconditional global well-poseness and scattering for the large data is an open problem. Kenig-Merle first solved the following conditional global well-posedness and scattering for 3D cubic NLS with general data, which is now called the critical norm conjecture.
\begin{conjecture}[Critical norm, Kenig-Merle\cite{KeMeR3GWP2010}]
	Let $d\geq 1$ and $ \alpha > 1$ such that $s_c \geq0$. Let $u : I \times \R^d\to\Bbb C$ be
	a maximal-lifespan solution to (\ref{general-NLS}) such that $u \in L_t^\infty\dot{H}_x^{s_c}
	(I\times\R^d)$. Then $u$ is global  and scatters to a free solution.
\end{conjecture}
For the inter-critical case $0<s_c<1$, Jason\cite{Murphy2014,Murphy2014b}  proved the long-time Strichartz estimates in for the inter-critical regularity level and show the case that $s_c\in[\frac{1}{2},1)$ for $d=4,5$; $s_c\in[\frac{1}{2},\frac{3}{4})$ for $d=3$ and $s_c=\frac{1}{2}$ for all $d\geq5$. Later, by further assuming the radial initial data, he fills the gap when $d=3$ in \cite{Murphy2015}.  For $s_c\in(\frac{1}{2},1)$ in $d\geq4$, we refer to \cite{GaoMiaoYang2019,GaoZhao2019}. Inspired by the work of \cite{Dodson2016b} who solve the mass-critical problem in $d=2$, Yu\cite{Yu2021} proved the scattering for quintic NLS on $\R^2$. For $d=1$, it is still a challenging problem.

For the supercritical regime, Killip-Visan\cite{KillipVisan2010} first proved the partial result by overcoming the difficulties brought by the non-locality of $|\nabla|^{s_c}.$ Miao-Murphy-Zheng\cite{MiaoMurphyZheng2014} utilized the long-time Strichartz estimate in the energy-supercritical level to show the scattering for $s_c\in(1,\frac{3}{2})$ for $d=4$. Afterwards, by using the strategy in Dodson\cite{D5} who overcomes the logarithmic failure of the double Duhamel formula, Dodson-Miao-Murphy-Zheng\cite{Dodson2017} proved the case  $s_c=\frac{3}{2}$ in dimension four. For the case of $s_c>\frac{3}{2}$, Lu-Zheng\cite{LuZheng2017} proved the scattering with radial initial data. For higher dimensional case, Zhao\cite{Tengfei} proved the scattering for $s_c\geq1$ and $\alpha$ is an odd integer. For other topic associated to the energy-supercritical NLS, we refer to \cite{Zheng2016,Zheng2014}. 

When $M=\T^d$, Bourgain\cite{Bour1993,BD} established the local-in-time Strichartz estimates 
\begin{align*}
	\big\|e^{it\Delta_{\T^d}}P_Nf\big\|_{L_{t,x}^p([0,T]\times\T^d)}\lesssim N^{\frac{d}{2}-\frac{d+2}{p}+\varepsilon}\|P_Nf\|_{L^2(\T^d)}.
\end{align*}
We note that $T<\infty$ and there exists an additional derivative loss, which are main differences between the tori case and Euclidean case. Due to this derivative loss, the local well-posedness  for NLS on tori is much more complicated than that on Euclidean space. By introducing the $X^{s,b}$-space, one can obtain the local well-posedness for \eqref{general-NLS} in $H^s$ with $s>s_c$. In order to obtain the critical local well-posedness, Herr-Tataru-Tzvetkov\cite{HerrT32011} established the trilinear estimate on the associated atomic space. The method formulated in \cite{HerrT32011} is now the standard arugment to treat the local well-posedness for NLS posed on the tori/compact manifold  in the critical space. Lee\cite{Lee} combined the Bony paradifferential technique with the atomic space to deduce the local well-posedness for \eqref{general-NLS} in $H^s(\T^3)$ with $s\geq s_c$.  Inspired by the concentration-compactness developed in \cite{KM}, Ionescu-Pausader\cite{IPT3} first proved the global well-posedness for energy-critical NLS on $\T^3$. Later, they generalized it to the waveguide manifold setting in \cite{IPRT3}. We refer to \cite{Luo-JMPA,Z,Yue,Z2,Z3} for more results on the long-time dynamics for energy-critical  NLS on waveguides.

For the nonlinear Schr\"odinger equation on tori at the non-conserved regularity, there are several interesting problems that remains unsolved. It is worth to point out that the scattering cannot occurs since the geometry is compact, see \cite{Chabert,CKSTT2010}  for more details. Recently, Yu-Yue\cite{YuYue} proved the global well-posedness of the quintic NLS on $\T^2$ under the assumption of the critical norm. Our main aim of this paper is to study the global well-posedness of cubic NLS on $\T^3$. 
\subsection{Main results}
\begin{theorem}\label{thm}
	Let $u_0\in H^\frac{1}{2}(\T^3)$ and let $u:I\times\T^3\to\Bbb C$ be a solution to \eqref{cubic-intro} on its maximal lifespan such that
	\begin{align*}
		\sup_{t\in I}\|u(t)\|_{H^\frac{1}{2}(\T^3)}<\infty.
	\end{align*} 
	Then $u$ is global, i.e. $I=\R$. Moreover, the solution map $u_0\to u$ is continuous from $H^\frac{1}{2}(\T^3)$ to $ X^{\frac{1}{2}}([-T,T])$.  
\end{theorem}
\begin{remark}
	To aviod technicalities, we restrict ourselves to the cubic NLS on $\T^3$, but the critical norm also holds for \eqref{general-NLS} with $\alpha=2k$ and $k\geq1$. For the energy-supercritical case, one should notice that by using the Bony paradifferential technique similar to \cite{Lee}, one can obtain the nonlinear estimate and stability result in $Z$-norms. 
\end{remark}
\subsection{Outline of the proof}
In this part, we will summarize the main steps and ingredients in the
proof of Theorem \ref{thm}. In our proof, we will follow the concentration-compactness/rigidity argument developed in \cite{KM,KeMeR3GWP2010} and utilize the Euclidean scattering solution as the black-box as in \cite{IPT3,IPRT3}. We also use the atomic space $U^p$ and $V^p$ as introduced in \cite{HadacKp22009,HerrT32011}.

\textit{Step 1}. The local well-posedness and stability result.

In the first step, we will use the atomic $U^p$ and $V^p$ space to find the suitable functional space $X^s$ and $Y^s$ which is frequency-localized and can be regarded as the ``critical" Bourgain spaces. To do this, we modified the definition in \cite{HerrT32011} to suit our $H^\frac{1}{2}$-critical setting. To obtain the global well-posedness for large data, following  \cite{IPT3}, we define a weaker critical norm $Z$ of the
solution, which plays a same role to the $L_{t,x}^5$-norm in \cite{KeMeR3GWP2010}.  Then, we prove the bilinear
estimate for frequency-localized functions in this atomic-type spaces defined above, which
overcome the logarithmic loss appeared in the Strichartz estimates and hence helps to
close the Picard iteration in our local theory (Proposition \ref{prop-lwp-LWP}). Moreover, using
this $Z$-norm and the local well-posedness theory, we also obtain a stability result(Proposition \ref{prop-stability}). 

\textit{Step 2}. The existence of the minimal blow-up solution.

In this step, we will use the black-box trick and the concentration-compactness argument to derive the minimal blow-up solution to \eqref{cubic-intro}. In order to use the scattering result on $\R^3$ as the black-box, we need to compare the solution on $\R^3$ and that on $\T^3$. This kind of comparison is valid only
for the  data that highly-concentrated and on the very short time interval. When considering the time interval beyond this time interval,  we need to study the extinction lemma, which is crucial in the study of large data theory in the curved space, see \cite{IPT3}. Then combining this extinction lemma with stability theory and the properties of the scattering solution on $\R^3$, we can show that the solution to \eqref{cubic-intro} on tori can be compared to the linear evolution with initial data $T_N\phi^{\pm\infty}$ where $\phi^{\pm\infty}$ is the scattering state of the  solution to \eqref{cubic-intro} on $\R^3$. After these preparations, one can prove the linear and nonlinear profile decompositions.. Then by  contradiction,   the induction on energy method pioneered by Bouragin\cite{Bourgain1999} and the linear profile decomposition, we can obtain the existence of minimal blow-up solution. Indeed, one can prove that there exists at most one profiles which is Euclidean profile. Finally, we preclude the possible of such minimal blow-up solution by using the stability theorem directly.

\vspace{2ex}\textbf{Organization of the paper}. In Section 2, we collect the definition and properties for some type of function space and Strichartz estimate. In Section 3, we first prove the bilinear Strichartz estimates in $U^p$ and $V^p$ spaces. Then we show the local well-posedness and stability theory for \eqref{cubic-intro} in the critical space. In Section 4, we establish the linear profile decomposition for bounded sequences in $H^\frac{1}{2}(\T^3)$ and construct the nonlinear profile decomposition. In Section 5, we prove the existence of the critical element and show the impossibility of the critical element which completes the proof of Theorem \ref{thm}. 
	\subsection{Notations}
	In this paper, we use $A\lesssim B$ to mean that there exists a constant such that $A\leqslant CB$, where the constant is not depending on $B$. We will also use $s+$ or $s-$, which means that there exists a small positive number $\varepsilon$ such that it is equal to $s+\varepsilon$ or $s-\varepsilon$ respectively. We denote $\langle x\rangle$ by $(1+\left|x\right|^2)^\frac{1}{2}$.
\section{Preliminaries}\label{sec-pre}
In this section, we recall the Strichartz estimates on torus and give the definition and the properties of the atomic-type spaces.  

\subsection{Function spaces and linear Strichartz estimates}
In this subsection, we introduce the $X^s$ and $Y^s$ spaces based on the atomic spaces $U^p$ and $V^p$. These atomic spaces were originally developed by Tataru and now is the standard target spaces to study  PDEs in critical spaces, see \cite{HadacKp22009,HerrT32011,HerrStri2014}. For convenience, we denote $\mathcal{H}$ the separable Hilbert space on $\C$ and also denote $\mathcal{Z}$ the set of finite partitions
$-\infty = t_0<t_1<\ldots<t_{K} = \infty$ on the real line. Furthermore, we set $v(\infty) := 0$ for any function $v:\R \rightarrow \mathcal{H}$.

We adapt the definition in \cite{HerrT32011} in the following:

\begin{definition}[$U^p$ space \cite{HerrT32011}]
	Let $1\leq p < \infty$ and functions $u: \R \rightarrow \mathcal{H}$. For sequence $\{t_k\}_{k=0}^K \in \mathcal{Z}$ and a sequence of functions $\{\phi_k\}_{k=0}^{K-1} \subset \mathcal{H}$ with $\sum_{k=0}^{K-1} \|\phi_k\|_{\mathcal{H}}^p = 1$ and $\phi_0 = 0$. A $U^p$-atom If a piecewise defined function $a : \mathbb{R} \to \mathcal{H}$ has the form
	\begin{equation*}
		a = \sum_{k=1}^K 1_{[t_{k-1}, t_k)} \phi_{k-1},
	\end{equation*}
	then we refer it as an $U^p$-atom. Then, we can define $U^p(\mathbb{R}, \mathcal{H})$ spaces as follows
	\begin{align*}
		u = \sum_{j=1}^{\infty} \lambda_{j} a_j,\qquad  \text{for}\hspace{1ex} a_j \hspace{1ex} \text{are}\hspace{1ex}   U^p\hspace{1ex} \text{-atoms}\  ,\quad \{\lambda_j\}_j \in \ell^1,\quad \|u\|_{U^p}<\infty,
	\end{align*}
	where
	\begin{align*}
		\|u\|_{U^p} : = \inf \bigg\{\sum_{j=1}^{\infty} |\lambda_j| : u = \sum_{j=1}^\infty \lambda_j a_j,\  \lambda_j \in \mathbb{C}\ \text{and}\hspace{1ex} a_j \hspace{1ex} \text{an}\hspace{1ex} U^p\ \text{atom}\bigg\}.
	\end{align*}
	The notation $1_{I}$ denotes the indicator function over the time interval $I$.
\end{definition}

\begin{definition}[$V^p$ spaces \cite{HerrT32011}]
	Let $1\leq p < \infty$ and functions $v:\R\rightarrow \mathcal{H}$. We also further assume that $v(\infty):=0$ and $v(-\infty) := \lim_{t\to -\infty} v(t)$ exists. Then, we define $V^p(\R,\mathcal{H})$ space such that
	\begin{align*}
		\|v\|_{V^p} : = \sup_{\{t_k\}_{k=0}^K\in \mathcal{Z}} \bigg(\sum_{k=1}^K \|v(t_k)-v(t_{k-1})\|^p_\mathcal{H}\bigg)^\frac{1}{p},
	\end{align*} 
	is empty.
	Likewise, we denote $V_{-}^p$ the closed subspace of all $v \in V^p$ with $\lim_{t\to -\infty}v(t) = 0$. We also denote $V_{-, rc}^p$ means all right-continuous $V_{-}^p$ functions.
\end{definition}

\begin{proposition}[Embeding properties \cite{HerrT32011}]\label{fml-pre-emb}
	For $1\leq p \leq q < \infty$, we have
	\begin{equation}
		U^p(\mathbb{R}, \mathcal{H}) \hookrightarrow U^q(\mathbb{R}, \mathcal{H}) \hookrightarrow  L^{\infty}(\mathbb{R},\mathcal{H}),
	\end{equation}
	and
	\begin{equation}
		U^p(\mathbb{R}, \mathcal{H}) \hookrightarrow V^p_{-,rc} (\mathbb{R}, \mathcal{H}) \hookrightarrow U^q(\mathbb{R}, \mathcal{H}).
	\end{equation}
\end{proposition}
 
\begin{definition}[$U^p_{\Delta}H^s$ and $V^p_{\Delta}H^s$ spaces \cite{HerrT32011}]
	For $s\in \mathbb{R}$, we define the $U^p_{\Delta}H^s$ as the space of all functions $u : \mathbb{R}\to H^s(\mathbb{T}^3)$ such that
	\begin{equation*}
		\|u\|_{U^p_{\Delta} H^s} := \|e^{-it\Delta}u(t)\|_{U^p(\mathbb{R},H^s)}.
	\end{equation*}
	Similarly, we can define the $V^p_{\Delta}H^s$ as the space of all functions $u : \mathbb{R}\to H^s(\mathbb{T}^3)$ such that
	\begin{equation*}
		\|u\|_{V^p_{\Delta} H^s} := \|e^{-it\Delta}u(t)\|_{V^p(\mathbb{R},H^s)}.
	\end{equation*} 
\end{definition}

We need the following embedding properties of the $U^p$ and $V^p$ spaces.
\begin{proposition}[Embedding, Proposition 2.2, Proposition 2.4 and Corollary 2.6 in \cite{HadacKp22009}]\label{prop-pre-Embed}
	For $1\leq p \leq q < \infty$,
	\begin{align*}
		U_{\Delta}^p(\R, L^2) \hookrightarrow V_{\Delta}^p (\R, L^2) \hookrightarrow U_{\Delta}^q(\R, L^2) \hookrightarrow  L^{\infty}(\R,L^2) .
	\end{align*}
\end{proposition}

\begin{proposition}[Duality, Theorem 2.8 in \cite{HadacKp22009}]\label{prop-pre-dual}
	Let $DU_\Delta^p$ be the space of functions
	\begin{align*}
		DU_\Delta^p = \{(i\partial_t + \Delta) u : u \in U_\Delta^p \},
	\end{align*}
	and the $DU_\Delta^p = (V_\Delta^{p' })^{\ast}$, with $\frac{1}{p} + \frac{1}{p' } = 1$. Then for any interval $0 \in J$
	\begin{align*}
		\norm{\int_0^t e^{i(t-t' )\Delta}F(u)(t' )\, dt' }_{U_\Delta^p (J\times \R^d)} \lesssim \sup\bracket{\int_J \inner{ v, F}  \,  dt: \norm{v}_{V_\Delta^{p' } }=1 }.
	\end{align*}
\end{proposition}
We also require the following property, which can transfer the map from $L^2$ to $U_\Delta^p$. 
\begin{proposition}[Transfer Principle, Proposition 2.19 in \cite{HadacKp22009}]\label{prop-pre-transpri}
	Let 
	\begin{align*}
		T_0 : L^2 \times \dots \times L^2 \to L_{loc}^{1}
	\end{align*}
	be an m-linear operator. Assume that for some $1\leq p,\ q \leq \infty$
	\begin{align*}
		\norm{ T_0 (e^{it\Delta} \phi_1,\dots, e^{it\Delta} \phi_m)}_{L^p(\R, L_x^q)} \lesssim \prod_{i=1}^m\norm{\phi_i}_{L^2}.
	\end{align*}
	Then, there exists an extension $T : U_{\Delta}^p \times \dots \times U_{\Delta}^p \to L^p(\R, L_x^q)$ satisfying
	\begin{align*}
		\norm{T(u_1, \dots, u_m)}_{L^p(\R, L_x^q)} \lesssim \prod_{i=1}^m \norm{u_i}_{U_{\Delta}^p},
	\end{align*}
	and  such that $T(u_1, \dots , u_m) (t, \cdot) = T_0 (u_1(t), \dots , u_m(t))(\cdot)$, a.e.
\end{proposition}

Now we define the following two function spaces which will be crucial in establishing the local and global well-posedness.
\begin{definition}[$X^s$ and $Y^s$ spaces in \cite{HerrT32011}]\label{def-pre-Xs-Ys}
	For $s\in \mathbb{R}$, we define $X^s$ and $Y^s$ the spaces of all functions $u : \mathbb{R}\to H^s(\mathbb{T}^3)$ such that for every $n\in \mathbb{Z}^3$, the map $t\mapsto e^{it|n|^2}\widehat{u(t)}(n)$ is in $U^2(\mathbb{R}, \mathbb{C})$ and $V_{rc}^2(\mathbb{R}, \mathbb{C})$, respectively, with the norm
	\begin{align*}
		\norm{u}_{X^s} : = \parenthese{\sum_{\xi \in \Z^d} \inner{\xi}^{2s} \norm{\widehat{e^{-it\Delta}u(t)}(\xi)}_{U^2}^2}^{\frac{1}{2}},\\
		\norm{u}_{Y^s} : = \parenthese{\sum_{\xi \in \Z^d} \inner{\xi}^{2s} \norm{\widehat{e^{-it\Delta}u(t)}(\xi)}_{V^2}^2}^{\frac{1}{2}}.
	\end{align*}
	
	For intervals $I \subset \R$, we define $X^s (I)$, $s \in \R$ as the restriction norms of $X^s$; thus
	\begin{align*}
		X^s(I) : = \left\{ u \in C(I,H^s) : \norm{u}_{X^s (I)} : = \sup_{\substack{J\subseteq\R\\|J|\leq1}}  \inf_{v|_J (t) = u|_J (t)} \norm{v}_{{X}^s} < \infty\right\} .
	\end{align*}
	We can define the spaces $Y^s(I)$ in a similar way. Therefore, we have the embedding relation
	\begin{equation}\label{fml-pre-Embed-XY}
		U^2_{\Delta}H^s \hookrightarrow X^s \hookrightarrow Y^s \hookrightarrow V^2_{\Delta}H^s.
	\end{equation} 
	
\end{definition}

\begin{proposition}[\cite{HadacKp22009}]\label{prop-pre-Xs-Hs}
	Suppose $u :=  e^{it\Delta}\phi$, then for any $T>0$ we have
	\begin{equation*}
		\|u\|_{X^s([0,T])} \leq \|\phi\|_{H^s}.
	\end{equation*}
\end{proposition}
We denote the norm $N(I)$ by the $X^\frac{1}{2}$-norm of the Duhamel term.
\begin{definition}[Nonlinear norm $N(I)$]\label{def-pre-Nnorm}
	Let $I=[0, T]$, we define the nonlinear norm $N(I)$ as
	\begin{equation*}
		\|f\|_{N(I)} := \left\| \int_0^t e^{i(t-t')\Delta} f(t') dt'\right\|_{X^{\frac{1}{2}}(I)}.
	\end{equation*}
\end{definition}

We also need a critical $Z$-norm, which will play a similar role as $L^5_{t,x}$. 
\begin{definition}[$Z$-norm]\label{def-pre-Z}
	\begin{align*}
		\norm{u}_{Z(I)} : = \sup_{J \subset I , \abs{J} \leq 1} \parenthese{ \sum_{N\in2^{\Bbb Z}} N^{5-p} \norm{P_N u(t)}_{L_{t,x}^p (J \times \T^3)}^p}^{\frac{1}{p}} .
	\end{align*}
	where $p$ will be chosen in the proof of refined Strichartz estimates(Lemma \ref{lem-ep-Ref-Stri}).
\end{definition}

We can find that $Z$-norm is weaker than the critical norm $X^{\frac{1}{2}}$:
\begin{remark}\label{rmk Z<X}
	Moreover,
	\begin{align*}
		\norm{u}_{Z(I) } \lesssim \norm{u}_{X^{\frac{1}{2}}(I)},
	\end{align*}
	thus $Z$ is indeed a weaker norm.  In fact, using Definition \ref{def-pre-Z} and Strichartz estimates below, we see that
	\begin{align*}
		\sup_{J \subset I , \abs{J} \leq 1} \parenthese{ \sum_N N^{5 -p} \norm{P_N u(t)}_{L_{t,x}^p (J \times \T^3)}^p}^{\frac{1}{p}}  \lesssim \sup_{J \subset I , \abs{J} \leq 1} \parenthese{ \sum_N N^{\frac{p}{2} }  \norm{P_N u(t)}_{X^{0} (J \times \T^3)}^p}^{\frac{1}{p}}  \lesssim \norm{u}_{X^{\frac{1}{2}}(I)}.
	\end{align*}
\end{remark}

\subsection{Linear Strichartz estimates for frequency-localized data}
In this part, we collect the linear Strichartz estimates in certain function spaces  with the frequency-localized initial data. 
\begin{proposition}[Strichartz estimates \cite{Bour1993,BD}]
	\label{prop-pre-Stri}
	If $p>\frac{10}{3}$, then
	\begin{equation*}
		\|P_N e^{it\Delta} f\|_{L^p_{t,x}([-1,1]\times\T^3)} \lesssim_p
		N^{\frac{3}{2}-\frac{5}{p}} \|f\|_{L^2_x}
	\end{equation*}
	and
	\begin{equation*}
		\|P_C e^{it\Delta} f\|_{L^p_{t,x}([-1,1]\times\T^3)} \lesssim_p
		N^{\frac{3}{2}-\frac{5}{p}} \|f\|_{L^2_x}
	\end{equation*}
	where $C$ is a cube of side length N and $f\in L^2(\T^3)$.
\end{proposition}

By the transfer principle (Proposition \ref{prop-pre-transpri}) and Strichartz  estimate,
Lemma \ref{prop-pre-Stri}, we obtain the following corollary:
\begin{corollary}\label{cor-strichartz}
	If $p>\frac{10}{3}$, for any $v\in U^p_\Delta([-1, 1])$,
	\[
	\|P_{\le N} v\|_{L_{t,x}^p([-1,1]\times\T^3)} \lesssim_p N^{\frac{3}{2}-\frac{5}{p}}
	\|P_{\le N}v\|_{U^p_\Delta L^2}\lesssim \|P_{\le N}v\|_{Y^0([0,1])},
	\]
	and
	\[
	\|P_C v\|_{L^p([-1,1]\times\T^3)} \lesssim_p N^{\frac{3}{2}-\frac{5}{p}}
	\|v\|_{U^p_\Delta([-1, 1])},
	\]
	where $C$ is a cube of side length $N$.
\end{corollary}

We also need the following type Strichartz estimate, which is crucial in the proof of bilinear estimates in atomic space.

Let $\mathcal{R}_M (N)$ be the collection of all sets in $\Z^3$ which are given as the intersection of a cube of side length $2N$ with strips of width $2M$, that is the collection of all sets of the form
\begin{align*}
	(\xi_0 + [-N, N]^3) \cap [\xi \in \Z^2 : \abs{a \cdot \xi -A} \leq M]
\end{align*}
where $\xi_0 \in \Z^2$, $a =\xi_0 |\xi_0|^{-1} $ and $A \in \R$.
\begin{proposition}\label{prop Strichartz L^infty}
	For all $1 \leq M \leq N$ and $R \in \mathcal{R}_M (N)$, we have
	\begin{align*}
		\norm{P_R e^{it\Delta} \phi}_{L_{t,x}^{\infty} (\T \times \T^3)} \lesssim M^{\frac{1}{2}} N \norm{P_R \phi}_{L_x^2 (\T^3)}.
	\end{align*}
\end{proposition}
\begin{proof}
	For the fixed  time $t$, we have the following bound
	\begin{align*}
		\norm{P_R e^{it\Delta} \phi}_{L_{t,x}^{\infty} (\T^3)} \leq \sum_{\xi \in \Z^2 \cap R} \abs{\widehat{\phi} (\xi)} \leq (\# (\Z^3 \cap R))^{\frac{1}{2}} \norm{P_R \phi}_{L_x^2 (\T^3)}.
	\end{align*}
	Then the rectangle contains at most $\frac{N^2}{M^2}$ almost disjoint cubes of side length $M$, and each cube contains approximately $M^3$ lattice points. Hence, we obtain the upper bound
	\begin{align*}
		\# (\Z^2 \cap R) \lesssim MN^2 .
	\end{align*}
	So far, we have finished the proof.
\end{proof}

\begin{corollary}\label{cor-pre-Stri-M}
	Let $p > \frac{10}{3}$ and $0 < \kappa < \frac{1}{2} - \frac{5}{3p}$. For all $1 \leq M \leq N$ and $R \in \mathcal{R}_M (N)$ we have 
	\begin{align*}
		\norm{P_R e^{it\Delta} \phi}_{L_{t,x}^p (\T \times \T^3)} \lesssim N^{\frac{3}{2}-\frac{5}{p}} \bigg(\frac{M}{N}\bigg)^{\kappa} \norm{P_R \phi}_{ L_x^2 (\T^3)}.
	\end{align*}
\end{corollary}

\section{Local well-posedness and stability theory}
To obtain local well-posedness in critical space with large initial data, we need to establish a trilinear estimate to deal with nonlinearity $|u|^2u$. To get this, we introduce a frequency localized bilinear estimate first
\begin{lemma}[Bilinear estimates]\label{lem-bilinear}
	Assuming $|I|$ is a bounded interval and dyadic numbers $N_1\geq N_2 \ge 1$, then we get that
	\begin{equation}
		\|P_{N_1} u_1 P_{N_2} u_2\|_{L^2_{x,t}(\T^3\times I)}
		\lesssim N_2^{\frac{1}{2}} (\frac{N_2}{N_1}+\frac{1}{N_2})^\kappa \|P_{N_1} u_1\|_{Y^0(I)}
		\|P_{N_2}u_2\|_{Y^0(I)}
	\end{equation}
	for some $\kappa >0$.
\end{lemma}
\begin{proof}
	We restrict $P_{N_1} u_1$ to a cube $C\in\mathcal{R}(N_2)$ with side-length $N_2$, and find that the support of $\widehat{P_{N_2} u}$ intersects with finite number of cubes $C$. Hence, by the almost orthogonality it is sufficient to prove
	\begin{equation}\label{fml-bilinear-cube}
		\|P_C P_{N_1} u_1 P_{N_2} u_2\|_{L^2_{x,t}(\T^3\times I)}
		\lesssim N_2^{\frac{1}{2}} (\frac{N_2}{N_1}+\frac{1}{N_2})^\kappa \|P_{N_1} u_1\|_{Y^0(I)}
		\|P_{N_2}u_2\|_{Y^0(I)}.
	\end{equation}
	Then, by the interpolation [\cite{HadacKp22009}, Proposition 2.20] we can find that $\eqref{fml-bilinear-cube}$ follows from the two estimates
	\begin{equation}\label{fml-bilinear-U2}
		\|P_C P_{N_1} u_1 P_{N_2} u_2\|_{L^2_{x,t}(\T^3\times I)} \lesssim N_2^{\frac{1}{2}} (\frac{N_2}{N_1}+\frac{1}{N_2})^{\kappa} \|P_{N_1} u_1\|_{U^2_{\Delta}(I) L^2}
		\|P_{N_2}u_2\|_{U^2_{\Delta}(I) L^2} ,
	\end{equation}
	and
	\begin{equation}\label{fml-bilinear-U4}
		\|P_C P_{N_1} u_1 P_{N_2} u_2\|_{L^2_{x,t}(\T^3\times I)} \lesssim N_2^{\frac{1}{2}} \|P_{N_1} u_1\|_{U^4_{\Delta}(I) L^2}
		\|P_{N_2}u_2\|_{U^4_{\Delta}(I) L^2} .
	\end{equation}
	Indeed, by H\"older's inequality and Strichartz estimate $\eqref{cor-strichartz}$, we get
	\begin{align*}
		\norm{P_C P_{N_1} u_1 P_{N_2} u_2 }_{L_{t,x}^2(I \times \T^3)} & \leq \norm{P_C P_{N_1} u_1}_{L_{t,x}^4(I \times \T^3)} \norm{ P_{N_2} u_2}_{L_{t,x}^4(I \times \T^3)}  \\
		& \lesssim N_2^{\frac{1}{4}} \norm{P_{N_1} u_1}_{U_{\Delta}^4(I) L^2} N_2^{\frac{1}{4}} \norm{ P_{N_2} u_2}_{U_{\Delta}^4(I) L^2} \\
		& = N_2^{\frac{1}{2}}\norm{P_{N_1} u_1}_{U_{\Delta}^4(I) L^2} \norm{ P_{N_2} u_2}_{U_{\Delta}^4(I) L^2},
	\end{align*}
	then $\eqref{fml-bilinear-U4}$ follows.
	
	It remains to prove $\eqref{fml-bilinear-U2}$. We only need to consider the case that $N_1 \gg N_2$, since $\eqref{fml-bilinear-U4}$ and the embedding relation $U^2_{\Delta}L^2 \hookrightarrow U^4_{\Delta}L^2$ implies $\eqref{fml-bilinear-U2}$ immediately. By Proposition \ref{prop-pre-transpri}, it is enough to prove the corresponding estimate for solutions to the linear equation:
	\begin{align}\label{eq Interpolation22}
		\norm{P_C e^{it\Delta} \phi_1 e^{it\Delta} \phi_2}_{L_{t,x}^2(I \times \T^3)}  \lesssim N_2^{\frac{1}{2}} \left( \frac{N_2}{N_1} + \frac{1}{N_2} \right)^{\kappa} \norm{\phi_1}_{ L^2 (\T^3)}  \norm{\phi_2}_{L^2 (\T^3)}.
	\end{align}
	The initial data $\phi_1,\phi_2$ are frequency localized
	\begin{equation*}
		supp \widehat{\phi}_1 \subset \{ |\xi|\sim N_1 \},\quad supp \widehat{\phi}_2 \subset \{ |\xi|\sim N_2 \}.
	\end{equation*}
	We extend $e^{it\Delta} \phi_j, j=1,2$ to the real line, and then consider a Schwartz function $\psi$ which is frequency localized in $[0,2\pi]$, to define that
	\begin{equation*}
		v_j(t) := \psi(t) e^{it\Delta} \phi_j,\quad j=1,2.
	\end{equation*}
	Then, we prove a stronger estimate 
	\begin{equation*}
		\|P_C u_1 u_2\|_{L^2(\R\times \T^3)} \lesssim N_2^{\frac{1}{2}}\bigg( \frac{N_2}{N_1} + \frac{1}{N_2} \bigg)^{\kappa}.
	\end{equation*}
	
	Let $\xi_0$ be the center of $C$, then we can find
	\begin{equation*}
		C := (\xi_0 + [-N_2,N_2]^3),
	\end{equation*}
	where $|\xi_0| \sim N_1$. We partition $C = \cup R_k$ into almost disjoint strips of width $M = \max \{ \frac{N_2^2 }{N_1} ,1 \}$, which are orthogonal to $\xi_0$
	\begin{equation*}
		C = \bigcup_{k\in\Z:|k|\simeq N_1/N_2} R_k
	\end{equation*}
	where
	\begin{align*}
		R_k = \{ \xi \in C : |\xi\cdot a - Mk - |\xi_0||\le M \}, \quad a=\xi_0|\xi_0|^{-1}.
	\end{align*}
	We decompose $P_C v_1\cdot v_2$ as follows 
	\begin{align*}
		P_C v_1 \cdot v_2 = \sum_{k\in\Z:|k|\simeq N_1/N_2} P_{R_k} P_{C} v_1 \cdot v_2
	\end{align*}
	and claim that the summands are almost orthogonal in $L^2 (\T \times \T^3)$. We will find that the orthogonality comes from the time frequency instead of the spacial frequency. Indeed, for $(\tau_1,\xi_1) \in supp \widehat{P_{R_k} v_1}$ we have $\xi_1 \in R_k$ and $|\tau_1 + \xi_1^2| \le 1$. Furthermore, basic computation yield
	\begin{align*}
		\xi_1^2 = (\xi_1\cdot a)^2 + |\xi_1 - \xi_0|^2 - ((\xi_1-\xi_0)\cdot a)^2 = M^2 k^2 + O(M^2 k),
	\end{align*}
	which implies
	\begin{equation*}
		|\tau_1 + M^2k^2| = |\tau_1 + |\xi_1|^2| + O(M^2 |k|) = O(M^2 |k|).
	\end{equation*}
	We conclude that the time frequency of $P_{R_k} v_1$ is supported in the interval 
	\begin{equation*}
		[ -M^2k^2 - cM^2 |k| , -M^2k^2 + cM^2 |k| ],
	\end{equation*} 
	where constant $c>0$. We can find the time frequency of $v_2$ is supported in the interval $[-N_2^2,N_2^2] \subset [-M^2|k|,M^2|k|]$. Therefore, the functions $\{P_{R_k}v_1\cdot v_2\}_k$ are almost orthogonal in $L^2(\R\times M)$, that is
	\begin{align*}
		\| P_C v_1\cdot v_2 \|_{L^2}^2 \simeq \sum_{k\in\Z:|k|\simeq N_1/N_2} \|P_{R_k}v_1\cdot v_2\|_{L^2}^2.
 	\end{align*}
 	By H\"older's inequality and Corollary \ref{cor-pre-Stri-M}, we have
 	\begin{align*}
 		\|P_{R_k}v_1\cdot v_2\|_{L^2} \lesssim & \|P_{R_k}v_1\|_{L^4} \|v_2\|_{L^4}\\
 		\lesssim & N_2^{\frac{1}{4}} \bigg(\frac{M}{N_2}\bigg)^{\kappa} \|P_{R_k}\phi_1\|_{U^4_{\Delta}L^2} N_2^{\frac{1}{2}} \|\phi_2\|_{U^4_{\Delta}L^2}\\
 		\lesssim &  N_2^{\frac{1}{2}} \bigg(\frac{N_2}{N_1} + \frac{1}{N_2}\bigg)^{\kappa}\|P_{R_k}\phi_1\|_{U^4_{\Delta}L^2}  \|\phi_2\|_{U^4_{\Delta}L^2}.
 	\end{align*}
We  	notice that the last inequality comes from the fact $\frac{M}{N_2} \le \frac{N_2}{N_1} + \frac{1}{N_2}$.
	Therefore, we have finished the proof of this proposition.
\end{proof}

\subsection{Nonlinear estimate}
Before starting nonlinear estimate, we introduce a $Z^{\prime}$ norm first
\begin{align*}
	\norm{u}_{Z' (I)} = \norm{u}_{Z(I)}^{\frac{3}{4}} \norm{u}_{X^{\frac{1}{2}} (I)}^{\frac{1}{4}}.
\end{align*}

\begin{proposition}\label{prop-nonlinear}
	For $u_l\in X^{\frac{1}{2}}(I), k=1,2,3$ and $ I $ is a bounded interval, we have
	\begin{align}\label{fml-trilinear}
		\norm{  \prod_{l=1}^{3} \widetilde{u_l}  }_{N(I)}  \lesssim \sum_{\{i,j,k\} = \{1,2,3\}} \norm{u_i}_{X^{\frac{1}{2}} (I)} \norm{u_j}_{Z^{\prime}(I)} \norm{u_k}_{Z^{\prime}(I)},
	\end{align}
	where $\widetilde{u_l}=u_l$ or $\widetilde{u_l}=\overline{u_l}$ for $l=1,2,3$. Furthermore, if there exist $A,B>0$ such that $P_{>A}u_1=u_1,P_{>A}u_2=u_2$ and $P_{<B}u_3=u_3$, we also have
	\begin{equation}\label{fml-trilinear-2}
		\norm{  \prod_{l=1}^{3} \widetilde{u_l}  }_{N(I)}  \lesssim \|u_1\|_{X^{\frac{1}{2}}(I)}\|u_2\|_{Z^{\prime}(I)}\|u_3\|_{Z^{\prime}(I)} + \|u_2\|_{X^{\frac{1}{2}}(I)}\|u_1\|_{Z^{\prime}(I)}\|u_3\|_{Z^{\prime}(I)}.
	\end{equation}
\end{proposition}

\begin{proof}
	Suppose that $N_0,N_1,N_2,N_3$ are dyadic numbers. Without loss of generality, we further assume that $N_1\ge N_2\geq N_3$. By  duality, we have that
	\begin{align*}
		&\norm{  \prod_{l=1}^{3} \widetilde{u_l}  }_{N(I)} \lesssim \sup\limits_{\|u_0\|_{Y^{-\frac{1}{2}}}} \bigg| \int_{\T^3\times I} \overline{u_0} \prod_{l=1}^3 \widetilde{u_l} dxdt \bigg|\\
		\le & \sup\limits_{\|u_0\|_{Y^{-\frac{1}{2}}}} \sum_{N_0,N_1\ge N_2\ge N_3} \bigg| \int_{\T^3\times I} \overline{P_{N_0}u_0} \prod_{l=1}^3 \widetilde{P_{N_l}u_l} dxdt \bigg|.
	\end{align*}
	By the spatial frequency orthogonality, we further assume that $N_1\sim max\{N_0,N_2\}$. Hence, we need to consider the following two cases:
	\begin{enumerate}
		\item
		$N_0 \sim N_1\geq N_2\geq N_3$,
		\item
		$N_0\le N_2 \sim N_1 \ge N_3$.
	\end{enumerate}
	
	\textbf{Case 1: $N_0 \sim N_1\geq N_2\geq N_3$}: On one hand, by Cauchy-Schwarz's inequality and bilinear estimate Lemma \ref{lem-bilinear}, we can find 
	\begin{align}\label{fml-lwp-nonl-int1}
		&\bigg| \int_{\T^3\times I} \overline{P_{N_0}u_0} \prod_{l=1}^3P_{N_l}\widetilde{u_l} dxdt  \bigg| \lesssim \|P_{N_0}u_0 P_{N_2}u_2\|_{L^2_{t,x}}\|P_{N_1}u_1 P_{N_3}u_3\|_{L^2_{t,x}}\nonumber\\
		\lesssim & \left( \frac{N_3}{N_1} + \frac{1}{N_3} \right)^{\kappa} \left( \frac{N_2}{N_0} + \frac{1}{N_2} \right)^{\kappa}\|P_{N_0}u_0\|_{Y^0(I)}\|P_{N_1}u_1\|_{Y^0(I)}\|P_{N_2}u_2\|_{X^{\frac{1}{2}}(I)}\|P_{N_3}u_3\|_{X^{\frac{1}{2}}(I)}.
	\end{align}
	On the other hand, we set $\{ C_j \}_{j \in \Z}$ the cube partition of size $N_2$ and $\{ C_k \}_{k \in \Z}$ the cube partition of size $N_3$. We can verify that $\{P_{C_j} P_{N_0}u_0 P_{N_2}u_2\}_j$ and $\{P_{C_k} P_{N_1}u_1 P_{N_3}u_3\}_k$ are both almost orthogonal in $L^2(\T^3\times I)$. Therefore, by H\"older's inequality and Strichartz estimates Lemma \ref{prop-pre-Stri}, we obtain 
	\begin{align}\label{fml-lwp-nonl-int2}
		& \bigg| \int_{\T^3\times I} \overline{P_{N_0}u_0} \prod_{l=1}^3P_{N_l}\widetilde{u_l} dxdt  \bigg| \le \|P_{N_0}u_0 P_{N_2}u_2\|_{L^2_{t,x}}\|P_{N_1}u_1 P_{N_3}u_3\|_{L^2_{t,x}}\nonumber \\
		& \lesssim \big( \sum_j \norm{(P_{C_j} P_{N_0} u_0) P_{N_2} u_2}_{L_{t,x}^2(I \times \T^3)}^2 \big)^{\frac{1}{2}} \big( \sum_k \norm{(P_{C_k} P_{N_1} u_1) P_{N_3} u_3}_{L_{t,x}^2(I \times \T^3)}^2 \big)^{\frac{1}{2}}\nonumber \\
		& \lesssim \big( \sum_j \norm{P_{C_j} P_{N_0} u_0}_{L_{t,x}^4(I \times \T^3)}^2 \norm{P_{N_2}u_2}_{L_{t,x}^4(I \times \T^3)}^2 \big)^{\frac{1}{2}}  \big( \sum_j \norm{P_{C_k} P_{N_1} u_1}_{L_{t,x}^4(I \times \T^3)}^2 \norm{P_{N_3}u_3}_{L_{t,x}^4(I \times \T^3)}^2 \big)^{\frac{1}{2}}\nonumber \\
		&\lesssim \big( \sum_j \norm{P_{C_j} P_{N_0} u_0}_{Y^0(I)}^2 N_2^{\frac{1}{2}}\norm{P_{N_2}u_2}_{L_{t,x}^4(I \times \T^3)}^2 \big)^{\frac{1}{2}}  \big( \sum_j \norm{P_{C_k} P_{N_1} u_1}_{Y^0(I)}^2 N_3^{\frac{1}{2}}\norm{P_{N_3}u_3}_{L_{t,x}^4(I \times \T^3)}^2 \big)^{\frac{1}{2}} \nonumber\\
		& \lesssim  \norm{P_{N_0} u_0}_{Y^0(I)}\norm{P_{N_1} u_1}_{Y^0(I)} \norm{P_{N_2} u_2}_{Z(I)}\norm{P_{N_3} u_3}_{Z(I)}.
	\end{align}
	Interpolating between $\eqref{fml-lwp-nonl-int1}$ and $\eqref{fml-lwp-nonl-int2}$, for some $\kappa_1 > 0$ we get
	\begin{align}\label{fml-lwp-nonlinear}
		&\bigg| \int_{\T^3\times I} \overline{P_{N_0}u_0} \prod_{l=1}^3P_{N_l}\widetilde{u_l} dxdt  \bigg| \le \|P_{N_0}u_0 P_{N_2}u_2\|_{L^2_{t,x}}\|P_{N_1}u_1 P_{N_3}u_3\|_{L^2_{t,x}}\nonumber\\
		\lesssim & \big( \frac{N_3}{N_1} + \frac{1}{N_3} \big)^{\kappa_1} \big( \frac{N_2}{N_0} + \frac{1}{N_2} \big)^{\kappa_1}\|P_{N_0}u_0\|_{Y^{-\frac{1}{2}}(I)}\|P_{N_1}u_1\|_{X^{\frac{1}{2}}(I)}\|P_{N_2}u_2\|_{Z^{\prime}(I)}\|P_{N_3}u_3\|_{z^{\prime}(I)}.
	\end{align}
	The coefficient $\big( \frac{N_3}{N_1} + \frac{1}{N_3} \big)^{\kappa_1} \big( \frac{N_2}{N_0} + \frac{1}{N_2} \big)^{\kappa_1}$ makes it possible to sum over $N_0\sim N_1 \ge N_2 \ge N_3$. In conclusion, we have
	\begin{align*}
		&\sum_{N_0\sim N_1 \ge N_2 \ge N_3} \left( \frac{N_3}{N_1} + \frac{1}{N_3} \right)^{\kappa_1} \left( \frac{N_2}{N_0} + \frac{1}{N_2} \right)^{\kappa_1}\|P_{N_0}u_0\|_{Y^{-\frac{1}{2}}(I)}\|P_{N_1}u_1\|_{X^{\frac{1}{2}}(I)}\|P_{N_2}u_2\|_{Z^{\prime}(I)}\|P_{N_3}u_3\|_{z^{\prime}(I)} \\
		\lesssim & \|u_0\|_{Y^{-\frac{1}{2}}(I)}\|u_1\|_{X^{\frac{1}{2}}(I)}\|u_2\|_{Z^{\prime}(I)}\|u_3\|_{z^{\prime}(I)}.
	\end{align*}
	
	\textbf{Case 2: $N_0\le N_2 \sim N_1 \ge N_3$}: Similar argument yields
	\begin{align}\label{fml-lwp-non-int3}
		&\bigg| \int_{\T^3\times I} \overline{P_{N_0}u_0} \prod_{l=1}^3P_{N_l}\widetilde{u_l} dxdt  \bigg| \lesssim \|P_{N_0}u_0 P_{N_2}u_2\|_{L^2_{t,x}}\|P_{N_1}u_1 P_{N_3}u_3\|_{L^2_{t,x}}\nonumber\\
		\lesssim & \left( \frac{N_3}{N_1} + \frac{1}{N_3} \right)^{\kappa} \left( \frac{N_2}{N_0} + \frac{1}{N_0} \right)^{\kappa}\|P_{N_0}u_0\|_{Y^0(I)}\|P_{N_1}u_1\|_{Y^0(I)}\|P_{N_2}u_2\|_{X^{\frac{1}{2}}(I)}\|P_{N_3}u_3\|_{X^{\frac{1}{2}}(I)},
	\end{align}
	and
	\begin{align}\label{fml-lwp-non-int4}
		& \bigg| \int_{\T^3\times I} \overline{P_{N_0}u_0} \prod_{l=1}^3P_{N_l}\widetilde{u_l} dxdt  \bigg| \le \|P_{N_0}u_0 P_{N_2}u_2\|_{L^2_{t,x}}\|P_{N_1}u_1 P_{N_3}u_3\|_{L^2_{t,x}}\nonumber \\
		& \lesssim  \norm{P_{N_0} u_0}_{Y^0(I)}\norm{P_{N_1} u_1}_{Y^0(I)} \norm{P_{N_2} u_2}_{Z(I)}\norm{P_{N_3} u_3}_{Z(I)}.
	\end{align}
	Interpolating between $\eqref{fml-lwp-non-int3}$ and $\eqref{fml-lwp-non-int4}$, summing over $N_0\le N_2 \sim N_1 \ge N_3$, we get
	\begin{align*}
		&\sum_{N_0\le N_2 \sim N_1 \ge N_3} \left( \frac{N_3}{N_1} + \frac{1}{N_3} \right)^{\kappa_1} \left( \frac{N_2}{N_0} + \frac{1}{N_2} \right)^{\kappa_1}\|P_{N_0}u_0\|_{Y^{-\frac{1}{2}}(I)}\|P_{N_1}u_1\|_{X^{\frac{1}{2}}(I)}\|P_{N_2}u_2\|_{Z^{\prime}(I)}\|P_{N_3}u_3\|_{z^{\prime}(I)} \\
		\lesssim & \|u_0\|_{Y^{-\frac{1}{2}}(I)}\|u_1\|_{X^{\frac{1}{2}}(I)}\|u_2\|_{Z^{\prime}(I)}\|u_3\|_{z^{\prime}(I)}.
	\end{align*}
	We summarize the other five cases $N_1\ge N_3 \ge N_2$, $N_2\ge N_1 \ge N_3$, $N_2\ge N_3 \ge N_1$, $N_3\ge N_1 \ge N_2$ and $N_3\ge N_2 \ge N_1$ and obtain $\eqref{fml-trilinear}$.
	
	Particularly, if exist $A,B>0$ such that $P_{>A}u_1=u_1,P_{>A}u_2=u_2$ and $P_{<B}u_3=u_3$, we only need to consider when $N_1\ge N_2 \gtrsim N_3$ and $N_2 \ge N_1 \gtrsim N_3$ and $\eqref{fml-trilinear-2}$ follows.
\end{proof}
By combining the bilinear estimate above and the Picard iteration, we have the following local well-posedness for \eqref{cubic-intro}.
\begin{proposition}[Local well-posedness]\label{prop-lwp-LWP}
	\begin{enumerate}
		\item\label{1} For the given $E>0$,  if  $u_0\in H^{\frac{1}{2}}(\T^3)$ satisfying $\|u_0\|_{H^{\frac{1}{2}}(\T^3)} \le E$, and there exist $\delta_0(E) > 0$ such that
		\begin{equation*}
			\|e^{it\Delta}u_0\|_{Z^{\prime}(I)} < \delta
		\end{equation*}
		for some $0<\delta < \delta_0$ and bounded interval $0\in I$. Then, there exist a unique solution $u\in X^{\frac{1}{2}}(I)$ to $\eqref{cubic-intro}$ with initial data $u_0$ satisfies
		\begin{equation*}
			\|u - e^{it\Delta}u_0\|_{X^{\frac{1}{2}}} \lesssim \delta^2.
		\end{equation*}
		
		\item\label{2} If $u \in X^{\frac{1}{2}}(I)$ is a strong solution of $\eqref{cubic-intro}$ on some interval $I$ and enjoys the bound
		\begin{align*}
			\|u\|_{L_t^\infty H^\frac{1}{2}(I\times\T^3)}\leq E.
		\end{align*}
		Moreover, if
		\begin{equation*}
			\|u\|_{Z(I)} < +\infty,
		\end{equation*} 
		then $u$ can be extended as a strong solution to an interval $J$, where $\overline{I} \subset J$, and
		\begin{equation*}
			\|u\|_{X^{\frac{1}{2}}(J)} \le C\big( E(u) , \|u\|_{Z(I)} \big).
		\end{equation*}
	\end{enumerate}
\end{proposition}
\begin{proof}
We prove \eqref{1} first. Denote  $\Phi$ by the nonlinear mapping
	\begin{equation*}
		\Phi(u) := e^{it\Delta}u_0 - i \int_{0}^{t}e^{i(t-s)\Delta} u(s)|u(s)|^2 ds,
	\end{equation*}
	and the work space
	\begin{equation*}
		S:= \{ u\in X^{\frac{1}{2}}(I): \|u\|_{X^{\frac{1}{2}}(I)} \le 2 E, \|u\|_{Z^{\prime}(I)}< c \},
	\end{equation*}
	where constant $c>0$ is to be chosen later.
	
	If $u_1,u_2 \in S$, by the nonlinear estimate Proposition \ref{prop-nonlinear}, we can verify that
	\begin{align*}
		&\|\Phi(u_1) - \Phi(u_2)\|_{X^{\frac{1}{2}}(I)}\\
		\le& C\big( \|u_1\|_{X^{\frac{1}{2}}(I)} + \|u_2\|_{X^{\frac{1}{2}}(I)} \big)\big( \|u_1\|_{Z^{\prime}(I)} + \|u_2\|_{Z^{\prime}(I)} \big) \|u_1 - u_2\|_{X^{\frac{1}{2}}(I)}\\
		\le & 2CcE \|u_1 - u_2\|_{X^{\frac{1}{2}}(I)}.
	\end{align*}
	Furthermore, by Proposition \ref{prop-pre-Xs-Hs} and nonlinear estimates Proposition \ref{prop-nonlinear}, then for any $u\in S$ we get
	\begin{align*}
		\|u\|_{X^{\frac{1}{2}}(I)} \le & \|\Phi(0)\|_{X^{\frac{1}{2}}(I)} + \|\Phi(u) - \Phi(0)\|_{X^{\frac{1}{2}}(I)}\\
		\le & \|u_0\|_{H^{\frac{1}{2}}} + Cc^2E,
	\end{align*}
	and 
	\begin{align*}
		\|u\|_{Z^{\prime}(I)} \le & \|\Phi(0)\|_{Z^{\prime}(I)} + \|\Phi(u) - \Phi(0)\|_{Z^{\prime}(I)}\\
		\le & \delta + Cc^2E.
	\end{align*}
	Then we take $c = 2\delta$ and $\delta_0(E) = \frac{1}{8CE} \ll 1$. Therefore, $\Phi$ is a contraction map and maps $S$ to itself. Consequently, we obtain a unique solution $u\in X^{\frac{1}{2}}(I)$ with initial data $u_0$. Using Proposition \ref{prop-nonlinear}, we can verify that
	\begin{equation*}
		\|u - e^{it\Delta}u_0\|_{X^{\frac{1}{2}}(I)} \lesssim \|u\|_{X^{\frac{1}{2}}(I)}\|u\|_{Z^{\prime}(I)}^3 \le E \delta^3 \lesssim \delta^2. 
	\end{equation*} 
	Thus, we have finished the proof of the first part.
	
Next, we prove \eqref{2}. Without loss of generality, we assume that $I = (0,T)$ for some $T>0$. We take $\varepsilon > 0$ sufficient small and $T_1 > T-1$ such that 
	\begin{equation*}
		\|u\|_{Z^{\prime}(T_1,T)} < \varepsilon.
	\end{equation*}
	Then, we set
	\begin{equation*}
		h(s) := \bigg\| e^{i(t-T_1)\Delta}u(T_1) \bigg\|_{Z^{\prime}(T_1,T_1+s)},
	\end{equation*}
	and $h(s)$ is continuous function with $h(0) = 0$. In the proof of {\bf (1)}, we can find $h(s) \leq \delta_0$, which implies the following:
	\begin{align*}
		\norm{u(t) - e^{i (t- T_1)\Delta} u(T_1)}_{X^{\frac{1}{2}} (T_1 , T_1 +s)} \leq h(s)^2 .
	\end{align*}
	By Duhamel's principle, we also obtain
	\begin{align*}
		\norm{ e^{i (t- T_1)\Delta} u(T_1)}_{Z(T_1 , T_1 +s)} & \leq \norm{u}_{Z (T_1 , T_1 +s)} + C \norm{u(t) - e^{i (t- T_1)\Delta} u(T_1)}_{X^{\frac{1}{2}} (T_1 , T_1 +s)} \\
		& \leq \varepsilon + C h(s)^2 .
	\end{align*}
	Thus gives
	\begin{align*}
		h(s) \leq \norm{ e^{i (t- T_1)\Delta} u(T_1)}_{Z(T_1 , T_1 +s)}^{\frac{3}{4}} \norm{ e^{i (t- T_1)\Delta} u(T_1)}_{X^{\frac{1}{2}}(T_1 , T_1 +s)}^{\frac{1}{4}} \leq (\varepsilon + C h(s)^3)^{\frac{3}{4}} \norm{u(T_1)}_{H^{\frac{1}{2}}}^{\frac{1}{4}} .
	\end{align*}
	Therefore, if $\varepsilon >0$ is sufficiently small, by continuity argument, we get
	\begin{align*}
		h(s) \lesssim \varepsilon^{\frac{3}{4}} E^{\frac{1}{4}} \le \frac{1}{2} \delta_0 ,
	\end{align*}
	for all $s \leq T- T_1$. Consequently, there exist $T_2 > T$ such that
	\begin{align*}
		\norm{e^{i(t -T_1)\Delta} u(T_1)}_{Z' (T_1 ,T_2)} \leq \frac{3}{4} \delta_0  .
	\end{align*}
	By the conclusion we have proved in {\bf (1)}, we see that $u$ can be extended to  $(0, T_2)$. 
\end{proof}

	Next, we will show the stability of equation \eqref{cubic-intro}. 
\begin{proposition}[Stability]\label{prop-stability}
Assume that $I$ is an open bounded interval, $\mu \in \{0, 1\}$, and $\widetilde{u} \in X^{\frac{1}{2}} (I)$ satisfies the approximate equation
\begin{align}\label{fml-lwp-stability-app}
(i\partial_t + \Delta_{\T^3}) \widetilde{u} = \abs{\widetilde{u}}^2 \widetilde{u} + e \quad \text{ on } I \times \T^3.
\end{align}
We assume that 
\begin{align*}
\norm{\widetilde{u}}_{Z(I)} + \norm{\widetilde{u}}_{L_t^{\infty} (I , H^{\frac{1}{2}} (\T^3))} 
\leq M,
\end{align*}
for some $0\le M<+\infty$ and there exists $t_0 \in I$, $u_0 \in H^{\frac{1}{2}}(\T^3)$ such that
\begin{align*}
\norm{u_0 -\widetilde{u}(t_0)}_{H^{\frac{1}{2}} (\T^3)} + \norm{e}_{N(I)} \leq \varepsilon,
\end{align*}
where $0 < \varepsilon < \varepsilon_1$ and $0 < \varepsilon_1 = \varepsilon_1 (M) \leq 1$ is sufficiently small.
	
Then there exists a strong solution $u \in X^{\frac{1}{2}} (I)$ of the Schr\"odinger equation
\begin{align}\label{fml-lwp-NLS}
(i\partial_t + \Delta_{\T^3}) u = \abs{u}^2 u ,
\end{align}
with initial data $u(t_0) = u_0$. Furthermore, $u$ satisfies
\begin{align*}
\norm{u}_{X^{\frac{1}{2}}(I)} + \norm{\widetilde{u}}_{X^{\frac{1}{2}} (I)} \leq C(M) ,\\
\norm{u-\widetilde{u}}_{X^{\frac{1}{2}}(I)} \leq C(M)\varepsilon .
\end{align*}
\end{proposition}

\begin{proof}
This proof is adapted from Proposition 3.4 in \cite{IPT3} and Proposition 3.5 in \cite{IPRT3} and we will prove it in several steps:
\begin{enumerate}
\item [\textit{Step1}:] From the proof of Proposition \ref{prop-lwp-LWP}, there exists $\delta_1 = \delta_1 (M)$ such that if for some interval $t_0 \in J \subset I$
\begin{align*}
	\norm{e^{i (t- t_0)\Delta} \widetilde{u}(t_0)}_{Z' (J)} + \norm{e}_{N(J)} \leq \delta_1,
\end{align*} 
then there exists a unique solution $\widetilde{u}$ to \eqref{fml-lwp-stability-app} on $J$ and
\begin{align*}
	\norm{\widetilde{u} - e^{i(t-t_0) \Delta} \widetilde{u} (t_0)}_{X^{\frac{1}{2}} (J)} \leq \norm{e^{i(t-t_0) \Delta} \widetilde{u} (t_0)}_{Z'(J)}^2 + 2 \norm{e}_{N(J)}.
\end{align*}

\item [\textit{Step2}:] If there exists $\varepsilon_1 = \varepsilon_1 (M)$ such that
\begin{align}\label{fml-lwp-small1}
	\begin{aligned}
		& \norm{e}_{N(I_k)} \leq \varepsilon_1 , \\
		& \norm{\widetilde{u}}_{Z(I_k)} \leq \varepsilon \leq \varepsilon_1 ,
	\end{aligned}
\end{align}
hold on $I_k = ( T_k ,T_{k+1} )$, we claim that
\begin{align}\label{fml-lwp-claim}
	\begin{aligned}
		& \norm{e^{i(t-t_0) \Delta} \widetilde{u} (T_k)}_{Z' (I_k)} \leq C (1 +M ) ( \varepsilon + \norm{e}_{N(I_k)})^{\frac{1}{2}} ,\\
		&\norm{\widetilde{u}}_{Z' (I_k)} \leq C (1 +M ) ( \varepsilon + \norm{e}_{N(I_k)})^{\frac{1}{2}} .
	\end{aligned}
\end{align}
We denote $h(s) = \norm{e^{i(t-T_k) \Delta} \widetilde{u}(T_k)}_{Z' (T_k ,T_k +s)}$ and let $J_k = [ T_k, T') \subset I_k$ be the largest interval such that $h(s) \leq \frac{1}{2} \delta_1(M)$. On one hand, Duhamel principle gives
\begin{align*}
	\norm{e^{i(t - T_k) \Delta} \widetilde{u}(T_k)}_{Z(T_k ,T_k + s)} & \leq \norm{\widetilde{u}}_{Z(T_k, T_k +s)} +  \norm{\widetilde{u} - e^{i(t-T_k) \Delta} \widetilde{u}(T_k)}_{X^{\frac{1}{2}} (T_k, T_k +s)} \\
	& \leq \varepsilon + h(s)^3 + 2 \norm{e}_{N(I_k)}.
\end{align*}
On the other hand, direct computation shows
\begin{align*}
	h(s) & \leq \norm{e^{i(t - T_k) \Delta} \widetilde{u}(T_k)}_{Z(T_k ,T_k + s)}^{\frac{3}{4}} \norm{e^{i(t - T_k) \Delta} \widetilde{u}(T_k)}_{X^{\frac{1}{4}} (T_k ,T_k + s)}^{\frac{1}{2}} \\
	& \leq \bigg( {\varepsilon + h(s)^2 + 2 \norm{e}_{N(I_k)}} \bigg)^{\frac{3}{4}} M^{\frac{1}{4}} \\
	& \leq C (1 + M) (\varepsilon + \norm{e}_{N(I_k)})^{\frac{3}{4}} + C(1+M)h(s)^{\frac{3}{2}} .
\end{align*}
By taking $\varepsilon_1$ sufficiently small, then claim $\eqref{fml-lwp-claim}$ follows immediately from continuous argument.

\item [\textit{Step3}:] On time interval $I_k = (T_k ,T_{k+1})$, we assume that
\begin{align}\label{fml-lwp-small-S3}
	\begin{aligned}
		&\norm{e^{i(t -T_k) \Delta}\widetilde{u} (T_k)}_{Z' (I_k)} \leq \varepsilon \leq \varepsilon_0 ,\\
		&\norm{\widetilde{u}}_{Z' (I_k)} \leq \varepsilon \leq \varepsilon_0 ,\\
		&\norm{e}_{N(I_k)} \leq \varepsilon_0, 
	\end{aligned}
\end{align}
where $\varepsilon_0$ is sufficiently small. Hence, $\|\widetilde{u}\|_{X^{\frac{1}{2}}(I_k)}$ can be controlled as follows
\begin{align*}
	\norm{\widetilde{u}}_{X^{\frac{1}{2}}(I_k)} \leq \norm{e^{i(t-T_k) \Delta}\widetilde{u} (T_k)}_{X^{\frac{1}{2}}(I_k)}  + \norm{\widetilde{u} - e^{i(t-T_k) \Delta}\widetilde{u} (T_k)}_{X^{\frac{1}{2}}(I_k)}  \leq M+1 .
\end{align*}
From Proposition \ref{prop-lwp-LWP}, there exists time interval $J_u$ and a strong solution $u$ to \eqref{fml-lwp-NLS} on $J_u$ such that
\begin{align*}
	a_k := \norm{\widetilde{u} (T_k) -u(T_k)}_{H^{\frac{1}{2}} (\T^3)} \leq \varepsilon_0 .
\end{align*}
Let $J_k = [T_k ,T_k +s] \cap I_k \cap J_u$ be the maximal interval such that
\begin{align}\label{fml-lwp-diff}
	\norm{\omega}_{Z'(J_k)} \leq 10 C \varepsilon_0 \leq \frac{1}{10 (M+1)} ,
\end{align}
where $\omega : = u - \widetilde{u}$. Observing that $J_k$ exists and is nonempty, since $s \mapsto \norm{\omega}_{Z'(T_k , T_k +s)} $ is finite and continuous on $J_u$ and vanishes for $s=0$. We can verify that $\omega$ solves
\begin{align*}
	(i \partial_t  + \Delta ) \omega = \mu ( \abs{\widetilde{u} +\omega}^2 (\widetilde{u} +\omega) - \abs{\widetilde{u}}^2 \widetilde{u})-e .
\end{align*}
By Proposition \ref{prop-nonlinear}, we have
\begin{align}\label{fml-lwp-diff-X}
	\norm{\omega}_{X^{\frac{1}{2}} (J_k)} & \leq \norm{e^{i(t-T_k)\Delta} (u(T_k) - \widetilde{u}(T_k))}_{X^{\frac{1}{2}} (J_k)} + \norm{ \abs{\widetilde{u} +\omega}^2 (\widetilde{u} +\omega) - \abs{\widetilde{u}}^2 \widetilde{u}}_{N(J_k)} + \norm{e}_{N(J_k)} \notag\\
	& \leq \norm{u(T_k) - \widetilde{u}(T_k)}_{H^{\frac{1}{2}}} + C \norm{\omega}_{X^{\frac{1}{2}}(J_k)} \left( \norm{\widetilde{u}}_{X^{\frac{1}{2}}(J_k)} \norm{\widetilde{u}}_{Z'(J_k)}^2 +  \norm{\omega}_{X^{\frac{1}{2}}(J_k)} \norm{\omega}_{Z'(J_k)}^2\right) + \norm{e}_{N(J_k)} \notag\\
	& \leq \norm{u(T_k) - \widetilde{u}(T_k)}_{H^{\frac{1}{2}}} + C \varepsilon_0  \norm{\omega}_{X^{\frac{1}{2}}(J_k)} + \norm{e}_{N(J_k)},
\end{align}
where we have used $\eqref{fml-lwp-diff}$ in the last inequality. Thus, for sufficiently small $\varepsilon_0 > 0$
\begin{align*}
	\norm{\omega}_{Z'(J_k)} \leq C \norm{\omega}_{X^{\frac{1}{2}}(J_k)} \leq 4 C  (\norm{u(T_k) - \widetilde{u}(T_k)}_{H^{\frac{1}{2}}} + \norm{e}_{N(J_k)}) \leq 8 C \varepsilon_0 .
\end{align*}
Note that $J_k = I_k \cap J_u$ and \eqref{fml-lwp-diff-X} holds on $I_k \cap J_u$. Therefore, follows from Proposition \ref{prop-lwp-LWP}, the solution $u$ can be extended to the entire interval $I_k$, then inequalities \eqref{fml-lwp-diff} and  \eqref{fml-lwp-diff-X} also hold on entire $I_k$. 

\item [\textit{Step4}:] We take $\varepsilon_2 (M ) \leq \varepsilon_1 (M)$ sufficiently small and split $I$ into pieces
\begin{equation*}
	I = \bigcup_{k = 1}^K I_k,
\end{equation*}
where $K = O (\norm{\widetilde{u}}_{Z(I)} / \varepsilon_2)^6$. On each interval $I_k$, we can find 
\begin{align*}
	&\norm{\widetilde{u}}_{Z(I_k)} \leq \varepsilon_2 ,\\
	&\norm{e}_{N(I_k)} \leq \kappa \varepsilon_2 .
\end{align*}
Thus, we get $\eqref{fml-lwp-claim}$, $\eqref{fml-lwp-diff}$, $\eqref{fml-lwp-diff-X}$, $\eqref{fml-lwp-small-S3}$ and $\eqref{fml-lwp-small1}$ on $I_k$ for $1\le k \le K$. Then, Proposition \ref{prop-stability} follows.
\end{enumerate}
\end{proof}

\section{Euclidean profiles}
Our main goal in this section is to compare Euclidean and periodic solutions of both linear and nonlinear Schr\"odinger equations. For concentrated initial data, this comparison can be obtained in a very short time interval, it usually be called Euclidean window. However, beyond the Euclidean window, period solution can not be approached via proper transform to Euclidean solution. Thanks to the global well-posedness and scattering result, we can compare period solution with the linear Euclidean solution with initial data that are related to the Euclidean scattering data. In this process, we need a crucial lemma which controls the linear period solution with concentrated initial data.   
 
For a fixed radial function $\eta \in C^{\infty}_0(\R^3)$ with $\eta(x) = 1, |x| \le 1$ and $\eta(x) = 0, |x|\ge 2$. Then, we define
\begin{align}\label{fml-eprofile-def}
\begin{aligned}
Q_N \phi \in H^{\frac{1}{2}} (\R^3) , &\qquad (Q_N \phi)(x) = \eta(N^{-\frac{1}{2}} x) \phi (x),\\
\phi_N   \in H^{\frac{1}{2}} (\R^3),  & \qquad \phi_N (x)  = N^{\frac{1}{2}} (Q_N \phi) (Nx) ,\\
T_N \phi  \in H^{\frac{1}{2}}(\T^3) , & \qquad T_N \phi (y)  = \phi_N (\Psi^{-1} (y)),
\end{aligned}
\end{align}
where $\Psi : \{ x \in \R^3 : \abs{x} < 1 \} \to \T^3$, $\Psi (x) =x$ is the projection on the torus. Thanks to the cut off function $\eta(N^{-\frac{1}{2}} x)$, we can find $Q_N \phi$ is a compactly supported in the ball of radial $2N^{\frac{1}{2}}$, and equals to $\phi$ in the ball of radial $N^{\frac{1}{2}}$. We also note that $\phi_N$ is an $\dot{H}^{\frac{1}{2}}$-invariant rescaling of $Q_N \phi$ and $T_N \phi$ is obtained by transfer $\phi_N$ to a neighborhood of zero in $\T^3$. 

Now, we introduce $H^{\frac{1}{2}}(\R^3)$ global well-posedness of the defocusing cubic NLS.
\begin{theorem}[$H^{\frac{1}{2}}(\R^3)$ global well-posedness \cite{KeMeR3GWP2010}]\label{thm-ep-GWP}
	Suppose $v$ is a solution to Cauchy problem 
	\begin{equation}\label{cubic-R3}
		(i\partial_t +\Delta)v = v|v|^2, \qquad v(0) = \phi,
	\end{equation}
	with initial data $\phi \in \dot{H}^{\frac{1}{2}}(\R^3)$ in a maximal lifespan of $I$. Let $T_{+}(u_0) := \sup\limits_{x}\{ x\in \R: x\in I \}$, and assume that 
	\begin{equation*}
		\sup\limits_{0 < t < T_{+}(u_0)} \|v(t)\|_{\dot{H}^{\frac{1}{2}}(\R^3)} = A < \infty.
	\end{equation*}
	Then $T_{+}(u_0) = +\infty$ and exist $\phi^{+}$ and $\phi^{-}$ such that
	\begin{equation*}
		\lim\limits_{t \to \pm\infty} \| v(t) - e^{it\Delta}\phi^{\pm} \|_{\dot{H}^{\frac{1}{2}}(\R^3)} = 0.
	\end{equation*}	
	Moreover, we have
	\begin{equation}
		\| |\nabla_{\R^3}|^{\frac{1}{2}} v\|_{(L_t^\infty L_x^2\cap L^2_tL_x^6)(\R\times\R^3)}
		\leq C(E_{\R^3}(\phi)) < +\infty.
	\end{equation}
	If initial data is sufficiently smooth, that is $\phi\in H^5(\R^3)$, we also have $v\in C(\R: H^5(\R^3))$ and
	\begin{equation}
		\sup_{t\in\R} \|v(t)\|_{H^5(\R^3)} \lesssim_{\|\phi\|_{H^5(\R^3)}} 1.
	\end{equation}
\end{theorem}

For fixed $\phi \in \dot{H}^{\frac{1}{2}}(\R^3)$, we consider solution on tori with concentrated initial data $T_N \phi$. For sufficiently large $T > 0$ and $N > 0$, we divide time interval $(-T,T)$ into the Euclidean window $(-TN^{-2} , TN^{-2})$ and beyond the Euclidean window $(-T,T)/(-TN^{-2} , TN^{-2})$. On the Euclidean window, the comparison is stated as follows:
\begin{lemma}\label{lem-window-profile}
	For fixed $\phi \in \dot{H}^{\frac{1}{2}} (\R^3)$ and $T_0 \in (0, \infty)$ are given. Then we have following conclusions:
	\begin{enumerate}
		\item 
		There exists $N_0 = N_0(\phi, T_0)$ sufficiently large such that for any $N \geq N_0$ there is a unique solution $U_N \in C ( (-T_0 N^{-2} , T_0 N^{-2}) , H^{\frac{1}{2}} (\T^3) )$ to the Cauchy problem
		\begin{align}\label{fml-ep-eqT}
			(i\partial_t + \Delta_{\T^3} ) U_N = \abs{U_N}^2 U_N, \quad U_N(0) = T_N \phi.
		\end{align}
		Moreover, for any $N \geq N_0$
		\begin{align*}
			\norm{U_N}_{X^{\frac{1}{2}} ((-T_0 N^{-2} , T_0 N^{-2}))} \lesssim  1,
		\end{align*}
		uniformly in $N$ and $T_0$
		
		\item 
		Given $\phi^{\prime} \in H^s(\R^3)$ with $s \ge 5$ and denote $v^{\prime} \in C(\R,H^s(\R^3))$ the solution to the Cauchy problem
		\begin{align*}
			(i\partial_t + \Delta_{\R^2}) v'  = \abs{v' }^2 v' , \quad v' (0)= \phi' .
		\end{align*}
		For $N \ge R \ge 1$ we define
		\begin{align*}
			\begin{aligned}
				v_R'  (t,x)  = \eta(\frac{x}{R}) v' (t,x), &\qquad  (t,x) \in (-T_0 , T_0) \times \R^3 ,\\
				v_{R, N}'  (t,x) = N v_R' (Nx, N^2 t), & \qquad (t,x) \in (-T_0N^{-2} , T_0 N^{-2}) \times \R^3, \\
				V_{R,N} (t,y)  = v_{R,N}'  (t , \Psi^{-1} (y) ), &\qquad  (t,y) \in  (-T_0N^{-2} , T_0 N^{-2}) \times \T^3.
			\end{aligned}
		\end{align*}
		Then, there exists $\varepsilon_1 = \varepsilon_1(E(\phi)) \in ( 0,1]$ is sufficiently small, such that for all $0 < \varepsilon < \varepsilon_1$ and $\phi'  \in H^s (\R^3)$ with $\norm{\phi - \phi' }_{\dot{H}^{\frac{1}{2}} (\R^3)} \leq \varepsilon_1$ we can find $R_0 = R_0(T_0,\phi^{\prime}) \ge 1$ such that for any $R \geq R_0$
		\begin{equation*}
			\limsup_{N \to \infty} \norm{U_N - V_{R,N}}_{X^{\frac{1}{2}} ((-T_0 N^{-2} , T_0 N^{-2}))} \lesssim \varepsilon_1 ,
		\end{equation*}
		uniformly in $N,R$ and $T_0$.
	\end{enumerate}
	
\end{lemma}

\begin{proof}
	It is sufficient to prove $(2)$, since stability result in Proposition \ref{prop-stability} and the fact that $V_{R,N}$ is an almost solution to $\eqref{fml-ep-eqT}$ implies $(1)$ immediately.  

	By Theorem \ref{thm-ep-GWP}, we have $v'$ exists globally and satisfies the following 
	\begin{align}\label{fml-ep-norm-es}
		\begin{aligned}
			\norm{ \abs{\nabla_{\R^3}}^{\frac{1}{2}} v'}_{L_t^{q} L_x^r  (\R \times \R^3)} \lesssim 1\\
			\sup_{t \in \R} \norm{v' (t)}_{H^5 (\R^3)} \lesssim_{ \norm{\phi'}_{H^5 (\R^3)}} 1 ,
		\end{aligned}
	\end{align}
	where $\frac{2}{q} = 3(\frac{1}{2} - \frac{1}{r})$.
	
	Note that $v^{\prime}(t,x)$ is the global solution to $\eqref{cubic-R3}$, hence $v'_{R} (t,x) = \eta(\frac{x}{R}) v'(t,x)$ solves the following approximate equation
	\begin{align*}
		(i \partial_t  + \Delta_{\R^2}) v'_{R} = \abs{v'_{R}}^2 v'_{R} + e_R (t,x) ,
	\end{align*}
	where
	\begin{align*}
		e_R (t,x) = \mu (\eta (\frac{x}{R}) - \eta^3 (\frac{x}{R}) ) \abs{v'}^2 v' + R^{-2} v'(t,x) (\Delta_{\R^2} \eta) (\frac{x}{R}) + 2R^{-1} \sum_{j=1}^2 \partial_j v'(t,x) \partial_j \eta (\frac{x}{R}) .
	\end{align*}
	Then $v'_{R,N} (t,x) $ solves
	\begin{align*}
		(i \partial_t  + \Delta_{\R^3}) v'_{R,N} =  \abs{v'_{R,N}}^2 v'_{R,N} + e_{R,N} (t,x) ,
	\end{align*}
	where
	\begin{align*}
		e_{R,N}(t,x) = N^3 e_R (N^2t, Nx) .
	\end{align*}
	Note that $V_{R,N} (t,y) = v'_{R,N} (t, \Psi^{-1} (y))$ if we take $N \ge 10R$. So, $V_{R,N}$ solves the following approximating equation
	\begin{align*}
		(i \partial_t  + \Delta_{\R^3}) V_{R,N} (t,y)=  \abs{V_{R,N}}^2  V_{R,N} + E_{R,N} (t,y) ,
	\end{align*}
	where
	\begin{align*}
		E_{R,N} (t,y) = e_{R,N} (t, \Psi^{-1} (y)).
	\end{align*}
	To apply Proposition $\ref{prop-stability}$, we need verity three conditions:
	\begin{enumerate}
		\item
		$\norm{V_{R,N}}_{L_t^{\infty} H_x^{\frac{1}{2}} ([-T_0 N^{-2},T_0 N^{-2}] \times \T^3)} + \norm{V_{R,N}}_{Z([-T_0 N^{-2},T_0 N^{-2}])} \lesssim 1$,
		\item
		$\norm{T_N \phi - V_{R,N}(0)}_{H_x^{\frac{1}{2}}} \leq \varepsilon$,
		\item
		$\norm{E_{R,N}}_{N([-T_0 N^{-2}, T_0 N^{-2}])} \leq \varepsilon$.
	\end{enumerate}
	
	\textbf{(1): $\norm{V_{R,N}}_{L_t^{\infty} H_x^{\frac{1}{2}} ([-T_0 N^{-2},T_0 N^{-2}] \times \T^3)} + \norm{V_{R,N}}_{Z([-T_0 N^{-2},T_0 N^{-2}])} \lesssim 1$} Thanks to the global existence of $v'(t,x)$, we can find that $V_{R,N}(t,y)$ also exists globally and satisfies
		\begin{align*}
			& \quad \sup_{t \in [-T_0 N^{-2} , T_0 N^{-2}]} \norm{V_{R,N} (t)}_{H_x^{\frac{1}{2}} (\T^3)} \leq \sup_{t \in [-T_0 N^{-2} , T_0 N^{-2}]} \norm{v'_{R,N} (t)}_{H_x^{\frac{1}{2}} (\T^3)} \\
			& = \sup_{t \in [-T_0 N^{-2} , T_0 N^{-2}]} \norm{N v'_R (N^2t, Nx)}_{H_x^{\frac{1}{2}} (\R^3)} \\
			& = \sup_{t \in [-T_0 N^{-2} , T_0 N^{-2}]} \frac{1}{N^{\frac{1}{2}}} \norm{v'_R (N^2t)}_{L_x^2 (\R^3)} + \norm{v'_R (N^2 t)}_{\dot{H}_x^{\frac{1}{2}} (\R^3)} \\
			& \leq \sup_{t \in [-T_0 , T_0 ]} \norm{v'_R}_{H_x^{\frac{1}{2}} (\R^3)} = \sup_{t \in [-T_0 , T_0 ]}  \norm{\eta (\frac{x}{R}) v'(t,x)}_{H_x^{\frac{1}{2}} (\R^3)}\\
			& \leq \sup_{t \in [-T_0 , T_0 ]} \norm{\eta(\frac{x}{R}) v'(t,x) }_{L_x^2(\R^3)} + \norm{\abs{\nabla}^{\frac{1}{2}} \eta(\frac{x}{R}) v'(t,x)}_{L_x^2 (\R^3)} + \norm{\eta (\frac{x}{R}) \abs{\nabla}^{\frac{1}{2}} v'(t,x)}_{L_x^2 (\R^3)} \\
			& \leq 2 \norm{v'(t,x)}_{H_x^{\frac{1}{2}} (\R^3)} \leq 2 \norm{\phi' }_{H^5 (\R^3)} .
		\end{align*}
		By Littlewood-Paley decomposition, Sobolev embedding and the fact that $\norm{v'(t)}_{L^{\infty}(\R,H_x^5)}$ is uniformly bounded, we obtain
		\begin{align*}
			\norm{V_{R,N}}_{Z([-T_0 N^{-2} ,T_0 N^{-2}])} & =  \sup_{J \subset [-T_0 N^{-2}, T_0 N^{-2}]} \parenthese{\sum_{M \, dyadic} M^{5 - p} \norm{P_M V_{R,N}}_{L_{t,x}^p(J \times \T^3)}^p}^{\frac{1}{p}} \\
			& \lesssim \sup_{J \subset [-T_0 N^{-2}, T_0 N^{-2}]} \norm{ \inner{\nabla}^{\frac{5}{p}-1} V_{R,N}}_{L_{t,x}^p(J \times \T^3)}\\
			& \lesssim \sup_{J \subset [-T_0 N^{-2}, T_0 N^{-2}]} \norm{ \inner{\nabla}^{\frac{1}{2}} V_{R,N}}_{L_t^p L_x^{\frac{6p}{3p-2}}(J \times \T^3)} \\
			& \lesssim \norm{\inner{\nabla}^{\frac{1}{2}} v^{\prime}_{R,N}}_{L_t^p L_x^{\frac{6p}{3p-2}}([-T_0 N^{-2}, T_0 N^{-2}] \times \R^3)} \\
			&\lesssim \sup_t \norm{v'(t)}_{H_x^5} \lesssim \norm{\phi^{\prime}}_{H^5_x},
		\end{align*}
		for any $p >\frac{10}{3}$. Then $(1)$ implies .

		\textbf{(2): $\norm{T_N \phi - V_{R,N}(0)}_{H_x^{\frac{1}{2}}} \leq \varepsilon$.} By definition and simple calculus, we get
		\begin{align*}
			\norm{T_N \phi - V_{R,N} (0)}_{H_x^{\frac{1}{2}} (\T^3)} & \leq \norm{\phi_N (\Psi^{-1} (y) - \phi'_{R,N} (\Psi^{-1} (y)}_{H_x^{\frac{1}{2}}(\T^3)} \\
			&\leq \norm{\phi_N -\phi'_{R,N}}_{\dot{H}_x^{\frac{1}{2}} (\T^3)}\\ 
			& = \norm{\eta(N^{-\frac{1}{2}} x) \phi (x) - \eta(N^{-\frac{1}{2}} x) \phi' (x) }_{\dot{H}_x^{\frac{1}{2}} (\R^3)} \\
			& \leq \norm{\eta(N^{-\frac{1}{2}} x) \phi (x) - \phi(x)}_{\dot{H}_x^{\frac{1}{2}} (\R^3)} + \norm{\phi - \phi'}_{\dot{H}_x^{\frac{1}{2}} (\R^3)} + \norm{\phi' - \eta(N^{-\frac{1}{2}} x) \phi' (x) }_{\dot{H}_x^{\frac{1}{2}} (\R^3)}.
		\end{align*}
		By taking $R > R_0$ with $R_0$ large enough and notice that $N \ge 10R$, we can make 
		\begin{equation*}
			\norm{\eta(N^{-\frac{1}{2}} x) \phi (x) - \phi(x)}_{\dot{H}_x^{\frac{1}{2}} (\R^3)} + \norm{\phi - \phi'}_{\dot{H}_x^{\frac{1}{2}} (\R^3)} + \norm{\phi' - \eta(N^{-\frac{1}{2}} x) \phi' (x) }_{\dot{H}_x^{\frac{1}{2}} (\R^3)},
		\end{equation*}
		sufficiently small.

		\textbf{(3): $\norm{E_{R,N}}_{N([-T_0 N^{-2}, T_0 N^{-2}])} \leq \varepsilon$.} By duality, we get
		\begin{align*}
			\norm{E_{R,N}}_{N([-T_0 N^{-2} ,T_0 N^{-2}])} & = \norm{\int_0^t e^{i(t-s) \Delta} E_{R,N} (s) \, ds}_{X^{\frac{1}{2}} ([-T_0 N^{-2} ,T_0 N^{-2}])} \\
			& \leq \sup_{\norm{u_0}_{Y^{-\frac{1}{2}}} =1} \abs{\int_{-T_0 N^{-2}}^{T_0N^{-2}} \int_{\T^3} \bar{u}_0 \cdot E_{R,N} \, dx dt} \\
			& \leq \sup_{\norm{u_0}_{Y^{-\frac{1}{2}}} =1} \norm{u_0}_{Y^{-\frac{1}{2}}} \norm{\abs{\nabla}^{\frac{1}{2}} E_{R,N}}_{L_t^1 L_x^2 ([-T_0 N^{-2} ,T_0 N^{-2}] \times \T^3)} \\
			& \leq \norm{\abs{\nabla_{\R^3} }^{\frac{1}{2}}e_{R,N}}_{L_t^1 L_x^2 ([-T_0 N^{-2} ,T_0 N^{-2}] \times \R^3)} = \norm{\abs{\nabla_{\R^3}}^{\frac{1}{2}} e_{R}}_{L_t^1 L_x^2 ([-T_0 ,T_0 ] \times \R^3)} .
		\end{align*}
		We claim that both
		\begin{equation}\label{fml-ep-L2-tend0}
			\lim_{R \to \infty} \norm{ e_R}_{L_t^1 L_x^2 ([-T_0, T_0] \times \R^3)} = 0,
		\end{equation}
		and
		\begin{equation}\label{fml-ep-H1-tend0}
			\lim_{R \to \infty} \norm{ |\nabla_{\R^3}| e_R}_{L_t^1 L_x^2 ([-T_0, T_0] \times \R^3)} = 0
		\end{equation}
		hold. Interpolating between then, we obtain
		\begin{equation*}
			\lim_{R \to \infty} \norm{ |\nabla_{\R^3}|^{\frac{1}{2}} e_R}_{L_t^1 L_x^2 ([-T_0, T_0] \times \R^3)} = 0,
		\end{equation*}
		which gives
		\begin{equation*}
			\lim_{R \to \infty} \norm{ |\nabla|^{\frac{1}{2}} E_{R,N}}_{L_t^1 L_x^2 ([-T_0N^{-2}, T_0N^{-2}] \times \T^3)} = 0.
		\end{equation*}
		Therefore, $(3)$ holds. 
		
		Then, we prove \eqref{fml-ep-L2-tend0} and \eqref{fml-ep-H1-tend0}. It is not hard to find $e_R$ has the bound:
		\begin{align*}
			|e_{R} (t,x)| & = \abs{   \mu (\eta (\frac{x}{R}) - \eta^3 (\frac{x}{R}) ) \abs{v'}^2 v' + R^{-2} v'(t,x) (\Delta_{\R^3} \eta) (\frac{x}{R}) + 2R^{-1} \sum_{j=1}^2 \partial_j v'(t,x) \partial_j \eta (\frac{x}{R}) }\\
			& \lesssim \abs{\mu (\eta (\frac{x}{R}) - \eta^3 (\frac{x}{R}) ) v'} + R^{-2} \abs{(\Delta_{\R^2} \eta) (\frac{x}{R}) v'} + R^{-1} \sum_{j=1}^2 \abs{\partial_j v'(t,x) \partial_j \eta (\frac{x}{R})},
		\end{align*}
		which implies $\eqref{fml-ep-L2-tend0}$. Note that $\partial_{x_k}e_R(t,x),k=1,2,3$ can be written as linear combination of the form
		\begin{equation*}
			1_{[R,2R]}(|x|)\bigg( |v^{\prime}(t,x)| + \sum_{i}^3|\partial_{x_i} v^{\prime}(t,x)| + \sum_{i}^3\sum_{j}^3|\partial_{x_i}\partial_{x_j} v^{\prime}(t,x)| \bigg),
		\end{equation*}
		therefore, $\eqref{fml-ep-H1-tend0}$ holds.

\end{proof}

We first introduce the extinction lemma, which shows the linear evolution beyond the Euclidean window can be arbitrarily small.
\begin{lemma}[Extinction Lemma]\label{lem-ep-extic} 
		Let $\phi\in\dot{H}^{\frac{1}{2}}(\R^3)$. For any $\varepsilon > 0$, there exists $T = T(\phi,\varepsilon) > 0$ and $N_0 = N_0(\phi,\varepsilon)$ such that for all $N > N_0$
		\begin{equation*}
			\|e^{it\Delta} T_N \phi\|_{Z(TN^{-2},T^{-1})} < \varepsilon.
		\end{equation*}
\end{lemma}
To prove the extinction lemma, we introduce the famous Dirichlet's lemma first
\begin{lemma}[Dirichlet's lemma]\label{lem-ep-Diri}
	For any real number $\alpha$, and positive integer $N$, there exist integers $p$ and $q$ such that $1 \leq q \leq N$ and $\abs{q \alpha - p} < \frac{1}{N}$.
\end{lemma}
\begin{proof}[Proof of Lemma \ref{lem-ep-extic}]
	For $M\ge1$, we define the kernel function
	\begin{equation*}
		K_M(t,x) := \sum_{\xi \in \Z^3} e^{ -[it|\xi|^2 + x\cdot \xi] }\eta(\xi/M) = e^{it\Delta}P_{\le M} \delta_0.
	\end{equation*}
	From lemma 3.18 of \cite{Bour1993}, $K_M$ enjoys the bound
	\begin{equation*}
		|K_M(t,x)| \lesssim \prod_{j=1}^3 \bigg( \frac{M}{\sqrt{q_j}} (1 + M\left|\frac{t}{2\pi} - \frac{a_j}{q_j}\right|^{\frac{1}{2}}) \bigg),
	\end{equation*}
	if $a_j,q_j$ satisfying $\frac{t}{2\pi} = \frac{a_j}{q_j} + \beta_j$, where $q_j \in \{1,2,\cdots,M\}, a_j\in \Z, (a_j,q_j) = 1$ and $|\beta_j| \le (Mq_j)^{-1}$ for all $j=1,2,3$.

	Let $|t|\le S^{-1}$, then Lemma \ref{lem-ep-Diri} gives $|\beta_j| \le (Mq_j)^{-1}\le M^{-1} \le S^{-1}$ for any $j=1,2,3$. So, we can find $\abs{\frac{a_j}{q_j}} \ge \frac{2}{S} \Rightarrow q_j \ge \frac{Sa_j}{2}$, that is either $a_j = 0$ or $q_j > \frac{S}{2}$ hold. 
	
	If $a_j \ne 0$, we get
	\begin{equation*}
		\frac{M}{\sqrt{q_j}} \left(1 + M\left|\frac{t}{2\pi} - \frac{a_j}{q_j}\right|^{\frac{1}{2}}\right) \lesssim \frac{M}{\sqrt{q_j}} \lesssim S^{-\frac{1}{2}} M.
	\end{equation*}
	
	If $a_j = 0$, we also have
	\begin{equation*}
		\frac{M}{\sqrt{q_j}} \left(1 + M\left|\frac{t}{2\pi} - \frac{q_j}{a_j}\right|^{\frac{1}{2}}\right) \lesssim \frac{M}{\sqrt{q_j}} + M|t|^{\frac{1}{2}} \lesssim |t|^{-\frac{1}{2}} \le S^{-\frac{1}{2}} M.
	\end{equation*}
	We conclude that for any $1 \le S \le M$
	\begin{equation}\label{fml-ep-K_M-Linf}
		\|K_M\|_{L^{\infty}_{t,x}((SM^{-2},S^{-1}))} \lesssim S^{-\frac{3}{2}}M^3.
	\end{equation}
	
	From definition \eqref{fml-eprofile-def}, we find that
	\begin{equation}\label{fml-ep-extin-int-bound}
	\begin{gathered}
		\|T_N \phi\|_{L^1(\T^3)} \lesssim _{\phi} N^{-\frac{5}{2}},\\
		\|P_{K} T_N \phi \|_{L^2(\T^3)} \lesssim _{\phi} N^{-\frac{1}{2}} \left( 1 + \frac{K}{N} \right)^{-5} \|\phi\|_{H^5_x(\R^3)} .
	\end{gathered}
	\end{equation}
	Combining those with Proposition \ref{prop-pre-Stri}, we obtain a refined Strichartz estimate
	\begin{equation}\label{fml-ep-Re-Str}
		\|e^{it\Delta}P_K T_N\phi\|_{L^p_{t,x}([-1,1]\times \T^3)} \lesssim K^{\frac{3}{2} - \frac{5}{p}} N^{-\frac{1}{2}} \left( 1 + \frac{K}{N} \right)^{-5} \|\phi\|_{H^5_x(\R^3)}.
	\end{equation}
	
	Then, we are devote to estimate $\|e^{it\Delta} T_N \phi\|_{Z(TN^{-2},T^{-1})}$ via $\eqref{fml-ep-Re-Str}$. By definition, we decompose $\|e^{it\Delta} T_N \phi\|_{Z(TN^{-2},T^{-1})}$ to three parts:
	\begin{align*}
	&	\|e^{it\Delta}T_N \phi\|_{Z([TN^{-2},T^{-1}])}^p =  \sup_{J\subset[TN^{-2},T^{-1}]}
		\sum_K K^{\frac{3}{2}p - 5} \|P_K e^{it\Delta}T_N \phi\|^p_{L^p(J\times\T^3)}  \\
		&=  \sup_{J\subset[TN^{-2},T^{-1}]} \bigg(\sum_{K\leq NT^{-\frac{1}{100}}} +
		\sum_{K\geq NT^{\frac{1}{100}}}+
		\sum_{NT^{-\frac{1}{100}}\leq K\leq NT^{\frac{1}{100}}} \bigg)
		K^{\frac{3}{2}p - 5} \|P_K e^{it\Delta}f_N\|^p_{L^p([TN^{-2}, T^{-1}]\times\T^3)}\\
		&= J_1 + J_2 + J_3.
	\end{align*}
	
First, apply $\eqref{fml-ep-Re-Str}$, we estimate term $J_1$ and $J_2$ as follows:
	\begin{align*}
		J_1 \lesssim & \sum_{K\leq NT^{-\frac{1}{100}}} K^{3p-10} \left(1+\frac{K}{N}\right)^{-5p}
		N^{-\frac{p}{2}} \|\phi\|_{H^{5}}\\
		\lesssim & N^{\frac{5p}{2} - 10} T^{\frac{3p - 10}{100}} \|\phi\|_{H^{5}},
	\end{align*}
	and
	\begin{align*}
		J_2 \lesssim & \sum_{K\leq NT^{-\frac{1}{100}}} K^{3p-10} \left(1+\frac{K}{N}\right)^{-5p}
		N^{-\frac{p}{2}}\|\phi\|_{H^{5}}\\
		\lesssim & \sum_{K\leq NT^{-\frac{1}{100}}} \bigg(\frac{K}{N}\bigg)^{-2p - 10} N^{3p - 10} \|\phi\|_{H^{5}}\\
		\lesssim & T^{-\frac{p + 5}{50}} \|\phi\|_{H^{5}}.
	\end{align*}
	
	Now, it remains to estimates $J_3$. We set $M \sim max(K,N)$ and $S \sim T$. Therefore, $K \in [NT^{-\frac{1}{100}} , NT^{\frac{1}{100}}]$ implies $M \le KT^{\frac{1}{100}} \le N T^{\frac{1}{50}}$. Then, we apply $\eqref{fml-ep-K_M-Linf}$ and Young's inequality to obtain
	\begin{equation}\label{fml-ep-inf}
		\begin{aligned}
			\|e^{it\Delta} P_K T_N \phi\|_{L^{\infty}_{t,x}([TN^{-2},T^{-1}] \times \T^3)} = & \|K_M * (T_N \phi)\|_{L^{\infty}_{t,x}([TN^{-2},T^{-1}] \times \T^3)} \\
			\lesssim & \|K_M\|_{L^{\infty}_{t,x}([TN^{-2},T^{-1}] \times \T^3)} \|T_N \phi\|_{L^1_{x}(\T^3)}\\
			\lesssim & T^{-\frac{3}{2}} M^3 N^{-\frac{5}{2}} \|\phi\|_{H^5} \le T^{-\frac{3}{2}+\frac{3}{50}} N^{\frac{1}{2}} \|\phi\|_{H^5}.
		\end{aligned}
	\end{equation} 
	We use $\eqref{fml-ep-Re-Str}$ to get
	\begin{equation}\label{fml-ep-10/3}
		\begin{aligned}
			\|e^{it\Delta} P_K T_N \phi\|_{L^{\frac{10}{3}+}_{t,x}([TN^{-2},T^{-1}] \times \T^3)} \lesssim & K^{\varepsilon} \bigg( 1 + \frac{K}{N} \bigg)^{-5} N^{-\frac{1}{2}} \\
			\lesssim & N^{-\frac{1}{2} + \varepsilon} T^{-\frac{\varepsilon}{100}} \|\phi\|_{H^5}.
		\end{aligned}
	\end{equation}
	Interpolating $\eqref{fml-ep-10/3}$ with $\eqref{fml-ep-inf}$, we have that
	\begin{equation*}
		\|e^{it\Delta} P_K T_N \phi\|_{L^{p}_{t,x}([TN^{-2},T^{-1}] \times \T^3)} \lesssim N^{\frac{1}{2} - \frac{5}{3p}} T^{-\frac{12(3p - 10)}{25p}} \|\phi\|_{H^5},
	\end{equation*} 
	which gives
	\begin{align*}
		J_3 \lesssim & \sum_{NT^{\frac{1}{100}} \le K\le NT^{\frac{1}{100}}} K^{\frac{3p}{2} - 5} \|e^{it\Delta} P_K T_N \phi\|_{L^{p}_{t,x}([TN^{-2},T^{-1}] \times \T^3)} \|\phi\|_{H^5} \\
		\lesssim & \sum_{NT^{\frac{1}{100}} \le K\le NT^{\frac{1}{100}}} K^{\frac{3p}{2} - 5} N^{\frac{p}{2} - \frac{10}{3}} \|\phi\|_{H^5} \\
		\lesssim & \sum_{NT^{\frac{1}{100}} \le K\le NT^{\frac{1}{100}}} \bigg( \frac{K}{N} \bigg)^{\frac{3p}{2} - 5} N^{2p - \frac{25}{3}} T^{-\frac{12(3p - 10)}{25}} \|\phi\|_{H^5} \\
		\lesssim & T^{-\frac{95(3p - 10)}{200}} \|\phi\|_{H^5} .
	\end{align*}
	Summing over $J_1,J_2$ and $J_3$, we have obtained the desired  estimate.
\end{proof}

We then introduce some notations and definitions of renormalized Euclidean frames. For fixed $f \in L^2(\T^3), t_0 \in \R$ and $x_0 \in \T^3$, we define
\begin{equation*}
	\begin{gathered}
		(\pi_{x_0} f)(x):= f(x-x_0),\\
		(\Pi_{t_0,x_0}) f(x):= (\pi_{x_0}e^{-it_0\Delta} f)(x).
	\end{gathered}
\end{equation*} 
\begin{definition}[Renormalized Euclidean frames]
	We define the set of renormalized Euclidean frames as follows
	\begin{align*}
		\widetilde{\mathcal{F}_e}:=&
		\{
		(N_k, t_k, x_k)_{k\geq 1}: N_k \ge 1, N_k \to +\infty \ t_k\to 0,\ x_k\in\T^3,\
		N_k\to\infty \\
		&\text{ and either } t_k=0 \text{ for any } k\geq 1 \text{ or }
		\lim_{k\to \infty} N_k^2|t_k| =\infty \}.
	\end{align*}
\end{definition}

\begin{proposition}[Euclidean profiles]\label{prop-ep-Ec-pro}
	Assume that $\mathcal{O} = (N_k, t_k, x_k)_k \in \widetilde{\mathcal{F}_e}$, $\phi \in \dot{H}^{\frac{1}{2}}(\R^3)$,
	$\phi\in \dot{H}^{\frac{1}{2}}(\R^3)$ and we set $U_k(0) := \Pi_{t_k,x_k}(T_{N_k}\phi)$. Then, we have the following properties 
	\begin{enumerate}
		\item For sufficient large $k = k(\mathcal{O},\phi)$, there exists $\tau =\tau(\phi)$ and
		$U_k\in X^{\frac{1}{2}}(-\tau, \tau)$ solve the initial value problem \eqref{cubic-intro} with initial data $U_k(0) = \Pi_{t_k, 0} (T_{N_k} \phi)$. Moreover, $U_k$ enjoys the bound
		\begin{equation}\label{fml-pro-Uk-bound}
			\|U_k\|_{X^{\frac{1}{2}}(-\tau, \tau)}\lesssim 1,
		\end{equation}
		uniformly in $k$.
		
		\item There exists $u\in C(\R: \dot{H}^{s}(\R^3))$ solve the equation
		\begin{equation}\label{fml-pro-NLS-R3}
			i\partial_t +\Delta_{\R^3}u = u|u|^2
		\end{equation}
		with scattering data $\phi^{\pm \infty}$ such
		that up to a subsequence: for any $\varepsilon >0$, there exists
		$T_0(\phi, \varepsilon)$ such that for all $T\geq T_0(\phi, \varepsilon)$, there exists
		$R_0(\phi, \varepsilon, T)$ such that for all $R\geq R_0(\phi, \varepsilon, T)$,
		there holds that
		\begin{equation}\label{fml-ep-in-wind}
			\|U_k -\widetilde{u_k}\|_{X^{\frac{1}{2}}(\{|t-t_k|\leq TN_k^{-2}\}\cap\{|t|<T^{-1}\})}
			\leq \varepsilon,
		\end{equation}
		for $k$ large enough, where
		\begin{equation}
			(\pi_{-x_k} \widetilde{u_k})(x, t) = N_k
			\eta(N_k \Psi^{-1}(x)/R) u(N_k\Psi^{-1}(x), N_k^2(t-t_k)).
		\end{equation}
		In addition, up to a subsequence,
		\begin{equation}\label{fml-ep-out-wind}
			\|U_k(t) -\Pi_{t_k-t,x_k}
			T_{N_k}\phi^{\pm\infty}\|_{X^{\frac{1}{2}}(\{ \pm(t-t_k) \geq TN_k^{-2}\}\cap\{|t|<T^{-1}\})}
			\leq \varepsilon,
		\end{equation}
		for $k = k(\phi,\varepsilon,T,R)$ large enough.
	\end{enumerate}
\end{proposition}
\begin{proof}
	We begin with a {special case that $t_k = 0$}. Given $0 < \varepsilon_0 \ll \varepsilon$, we can find $\phi^{\prime} \in \dot{H}^{s_0}(\R^3)$ such that
	\begin{equation*}
		\| \phi - \phi^{\prime} \|_{\dot{H}^{s_0}(\R^3)} < \varepsilon_0 \ll \varepsilon,
	\end{equation*} 
	where $s_0 = max\{s,5\}$.
	By the global well-posedness result in Theorem \ref{thm-ep-GWP}, there exists $u \in C(\R,H^{s_0}(\R^3))$ solves \eqref{cubic-R3} with initial data $\phi^{\prime}$ and scattering data $\phi^{\pm} \in \dot{H}^{\frac{1}{2}}(\R^3)$. Let $T > 0$ be fixed. We take $\varepsilon_0 = \varepsilon_0(\phi , \varepsilon)$ satisfies Lemma \ref{lem-ep-decomp}, then for $R > R_0 = R_0(\phi,\varepsilon,T)$ and $k > k_0 = k_0(\phi,\varepsilon,T,R)$ there exists unique
	\begin{equation*}
		U_k \in C\big((-2TN_k^{-2} , 2TN_k^{-2}), \dot{H}^{\frac{1}{2}}(\T^3)\big) \cap X^{\frac{1}{2}}(-2TN_k^{-2} , 2TN_k^{-2})
	\end{equation*}  
	solves \eqref{cubic-intro}. Moreover, $U_k$ enjoys the bound
	\begin{equation}\label{fml-ep-Ukuk-small}
		\| U_k - \widetilde{u}_k \|_{X^{\frac{1}{2}}(-2TN_k^{-2} , 2TN_k^{-2})} \lesssim \varepsilon_0 .
	\end{equation}
	We have accomplished \eqref{fml-ep-in-wind}. 
	
	Now, we turn to the estimating beyond the Euclidean window, that is \eqref{fml-ep-out-wind}. Thanks to the extinction lemma, for $T \ge T_0 = T_0(\phi,\varepsilon_0)$ and $k \ge k_0 = k_0(\phi , \varepsilon_0)$,
	\begin{equation}\label{fml-ep-extinc}
		\| e^{it\Delta} \phi^{+\infty} \|_{Z(-TN_k^{-2} , T^{-1})} < \varepsilon_0.
	\end{equation}
	We claim that
	\begin{equation}\label{fml-ep-claim}
		\| e^{i(t - T_0N_k^{-2})\Delta} U_k(T_0 N^{-2}_k) \|_{Z(T_0N_k^{-2},T_0^{-1})} \lesssim \varepsilon_0.
	\end{equation}
	Assuming claim \eqref{fml-ep-claim}, then using Proposition \ref{prop-lwp-LWP} insures there exist unique solution $U_k$ to \eqref{cubic-intro} on the time interval $(-T_0^{-1},T_0^{-1})$. 
	
	We prove claim \eqref{fml-ep-claim}. Let $T \ge T_0$, $R \ge R_0$ and we denote $I_k := (TN_k^{-2} , T^{-1})$ and $J_k := (0 , T^{-1} - TN_k^{-2} )$. We deduce claim \eqref{fml-ep-claim} to the following
	\begin{align*}
		\| e^{it\Delta} U_k(TN_k^{-2}) \|_{Z(J_k)} \le & \| e^{it\Delta} \big( U_k(TN_k^{-2}) - \widetilde{u}_k(TN_k^{-2}) \big) \|_{Z(J_k)}\\
		& + \| e^{it\Delta} \big( \widetilde{u}_k(TN_k^{-2}) - e^{iTN_k^{-2}\Delta}(T_{N_k}\phi^{+\infty}) \big) \|_{Z(J_k)}\\
		& + \|e^{it\Delta}\|_{Z(I_k)} \\
		:= & S_1 + S_2 + S_3.
	\end{align*} 
	Combining \eqref{fml-ep-Ukuk-small} with Proposition \ref{prop-lwp-LWP}, $S_1$ can be bounded as follows:
	\begin{equation*}
		S_1 \lesssim \| U_k(TN_k^{-2}) - \widetilde{u}_k(TN_k^{-2}) \|_{H^{\frac{1}{2}}(\T^3)} \lesssim \varepsilon_0.
	\end{equation*}
	Similarly, we get
	\begin{equation}\label{fml-ep-S3-deduce}
		S_2 \lesssim \| \widetilde{u}_k(TN_k^{-2}) - e^{iTN_k^{-2}\Delta} (T_{N_k}\phi^{+\infty}) \|_{H^{\frac{1}{2}}(\T^3)}.
	\end{equation}
	Denoting $\phi^{\prime\prime}\in H^{5}(\R^3)$ such that $\|\phi^{\prime\prime} - \phi^{+\infty}\|_{H^{\frac{1}{2}}(\R^3)} < \varepsilon_0$, then we use triangle inequality to \eqref{fml-ep-S3-deduce} and arrive at
	\begin{align*}
		S_2 \lesssim & \| \widetilde{u}_k(TN_k^{-2}) - V_{R,N_k}(e^{it\Delta}\phi^{+\infty})(TN_k^{-2}) \|_{H^{\frac{1}{2}}(\T^3)} \\
		+ & \| V_{R,N_k}(e^{it\Delta}\phi^{+\infty})(TN_k^{-2}) - V_{R,N_k}(e^{it\Delta}\phi^{\prime\prime})(TN_k^{-2}) \|_{H^{\frac{1}{2}}(\T^3)} \\
		+ & \| V_{R,N_k}(e^{it\Delta}\phi^{\prime\prime})(TN_k^{-2}) - e^{iTN_k^{-2}\Delta} (T_{N_k}\phi^{+\infty}) \|_{H^{\frac{1}{2}}(\T^3)}\\
		& := S_2^1 + S_2^2 + S_2^3.
	\end{align*}
	By the scattering property, there exist $T_0 = T_0(\phi,\varepsilon_0)$ large enough such that $S_2^1 \lesssim \varepsilon_0$ hold for any $T>T_0$. Thanks to $\|\phi^{\prime\prime} - \phi^{+\infty}\|_{H^{\frac{1}{2}}(\T^3)} < \varepsilon_0$, so we have $S_2^2 \lesssim \varepsilon_0$. Using Lemma \ref{lem-window-profile} we can get $S_2^3 \lesssim \varepsilon_0$. As an conclusion, we get $S_2 \lesssim \varepsilon_0$.
	
	The smallness of term $S_3$ is an application of \eqref{fml-ep-extinc}. So far, we have proved claim \eqref{fml-ep-claim}.
	
	It remains to verify \eqref{fml-ep-out-wind}, thanks to symmetry we only need to show
	\begin{equation*}
		\|U_k(t) -\Pi_{t_k-t,x_k}
		T_{N_k}\phi^{+\infty}\|_{X^{\frac{1}{2}}(\{ t-t_k \geq TN_k^{-2}\}\cap\{|t|<T^{-1}\})}
		\leq \varepsilon.
	\end{equation*}
	By triangle inequality, for $T \ge T_0$, $R\ge  R_0$ and $k\ge k_0$, we obtain
	\begin{align*}
		\|U_k(t) - e^{it\Delta}(T_{N_k}\phi^{+\infty})\|_{X^{\frac{1}{2}}(I_k)} \le & \|U_k(t) - e^{i(t - TN_k^{-2})\Delta}U_k(TN_k^{-2})\|_{X^{\frac{1}{2}}(I_k)} \\
		& + \|e^{i(t - TN_k^{-2})\Delta}(U_k(TN_k^{-2}) - \widetilde{u}_k(TN_k^{-2}))\|_{X^{\frac{1}{2}}(I_k)} \\
		& + \|e^{it\Delta}( e^{-itN_k^{-2}}\widetilde{u}_k(TN_k^{-2}) - T_{N_k}\phi^{+\infty} )\|_{X^{\frac{1}{2}}(I_k)}\\
		:= & \widetilde{S}_1 + \widetilde{S}_2 + \widetilde{S}_3.
	\end{align*}
	We use Proposition \ref{prop-lwp-LWP} to $\widetilde{S}_1$ and obtain
	\begin{equation*}
		\widetilde{S}_1 \lesssim \varepsilon_0^2 \lesssim \varepsilon.
	\end{equation*}
	Via the embedding relation \eqref{fml-pre-emb} and \eqref{fml-ep-Ukuk-small}, the second term can be bounded as follows 
	\begin{equation*}
		\widetilde{S}_2 \lesssim \|U_k(TN_k^{-2}) - \widetilde{u}_k(TN_k^{-2})\|_{H^{\frac{1}{2}}(\T^3)} \lesssim \varepsilon.
	\end{equation*}
	Similarly, we have
	\begin{equation*}
		\widetilde{S}_3 \lesssim \|\widetilde{u}_k(TN_k^{-2}) - e^{itN_k^{-2}\Delta}(T_{N_k}\phi^{+\infty})\|_{H^{\frac{1}{2}}(\T^3)}.
	\end{equation*}
	The same strategy as we have deal with \eqref{fml-ep-S3-deduce}, we get $\widetilde{S}_3 \lesssim \varepsilon$. Combining those together, we obtain \eqref{fml-ep-out-wind} for $t_k = 0$ and $k \ge 1$.
	
	Then we consider {the case that} \textbf{$\lim_{k\to \infty} N_k^2t_k = \infty$}. Without loss of generality, we assume that $\lim_{k\to \infty} N_k^2t_k = +\infty$. Thanks to the existence of wave operator and Proposition \ref{prop-lwp-LWP}, there exists $u \in C(\R,H^{\frac{1}{2}}(\R^3))$ such that
	\begin{equation*}
		\lim_{t \to -\infty} \| u(t) - e^{it\Delta}\phi \|_{H^{\frac{1}{2}}(\R^3)} = 0.
	\end{equation*}
	We set $u(0) := \widetilde{\phi}$, frame $\widetilde{\mathcal{O}}$ and associated profile $\{ T_{N_k}\widetilde{\phi} \}_k$. Therefore, we can apply the first case and get a unique $V_k$, which is a solution to \eqref{cubic-intro} on the time interval $(-T_0,T_0)$ with initial data $V_k(0) = T_{N_k}\widetilde{\phi}$. For sufficiently large $k$, \eqref{fml-ep-out-wind} implies
	\begin{equation*}
		\lim_{k \to \infty} \|V_k(-t_k) - \Pi_{t_k,0}T_{N_k}\phi\|_{H^{\frac{1}{2}}(\T^3)} \lesssim \lim_{k \to \infty} \|V_k(t) - \Pi_{-t,0}T_{N_k}\phi\|_{X^{\frac{1}{2}}(\{-t\ge T_0N_k^{-2}\}\cap \{|t|\le T_0^{-1}\})} = 0.
	\end{equation*} 
	We apply stability result in Proposition \ref{prop-stability} to it and get that
	\begin{equation*}
		\lim_{k \to \infty} \| V_k(\cdot - t_k) - U_k \|_{X^{\frac{1}{2}}(-T_0,T_0)} = 0.
	\end{equation*}
	Hence, $U_k$ inherits \eqref{fml-ep-in-wind} and \eqref{fml-ep-out-wind} from $V_k$ and we have proved this Proposition.
\end{proof}

From the proof of Proposition \ref{prop-ep-Ec-pro}, we can find that for a given $U_k$ it has different approximating in time intervals. Based on this fact, we introduce the following Corollary, which decompose the nonlinear Euclidean profile from Proposition \ref{prop-ep-Ec-pro}.

\begin{corollary}[Decomposition of the nonlinear Euclidean profiles $U_k$]\label{cor-ep-Dec-Npro}
	Let $U_k$ be the nonlinear Euclidean profiles with respect to $\mathcal{O} = (N_k , t_k, x_k)_k \in \widetilde{\mathcal{F}}_e$ in Proposition \ref{prop-ep-Ec-pro}. For any $\theta > 0$, there exists large $T_{\theta}^0$ and $R_{\theta}$, such that for all $T_{\theta} \geq T_{\theta}^0$ and $k = k(R_{\theta})$ large enough, $U_k$ can be decompose as follows:
	\begin{align*}
		\mathbf{1}_{(-T_{\theta}^{-1} , T_{\theta}^{-1})} (t) U_k = \omega_k^{\theta , -\infty} + \omega_k^{\theta , +\infty} + \omega_k^{\theta} + \rho_k^{\theta} ,
	\end{align*}
	and $\omega_k^{\theta , \pm \infty}$, $\omega_k^{\theta}$ and $\rho_k^{\theta}$ has the following properties:
	\begin{align}\label{fml-ep-ndecomp-proper}
		\begin{aligned}
			& \norm{\omega_k^{\theta, \pm \infty}}_{Z'  (-T_{\theta}^{-1} , T_{\theta}^{-1})} + \norm{\rho_k^{\theta}}_{X^{\frac{1}{2}} (-T_{\theta}^{-1} , T_{\theta}^{-1})} \leq \theta,\\
			& \norm{\omega_k^{\theta, \pm \infty}}_{X^{\frac{1}{2}} (-T_{\theta}^{-1} , T_{\theta}^{-1})} + \norm{\omega_k^{\theta}}_{X^{\frac{1}{2}} (-T_{\theta}^{-1} , T_{\theta}^{-1})} \lesssim 1,\\
			& \omega_k^{\theta , \pm \infty} = P_{\leq R_{\theta} N_k } \omega_k^{\theta , \pm \infty} ,\\
			& \abs{\nabla_x^m \omega_k^{\theta}} + (N_k)^{-2} \mathbf{1}_{S_k^{\theta}} \abs{\partial_t \nabla_x^m \omega_k^{\theta} } \leq R_{\theta} (N_k)^{\abs{m} +\frac{1}{2}} \mathbf{1}_{S_k^{\theta}}, \quad m \in \mathbb{N}^2, \, 0 \leq \abs{m} \leq 10,
		\end{aligned}
	\end{align}
	where $S_k^{\theta} : = \{ (t,x) \in  (-T_{\theta}^{-1} , T_{\theta}^{-1}) \times \T^3 : \abs{t - t_k} \leq T_{\theta} N_k^{-2} , \abs{x-x_k} \leq R_{\theta} N_k^{-1} \}$.
\end{corollary}

\begin{proof}
	By Proposition \ref{prop-ep-Ec-pro}, there exists $T(\phi, \theta)$, such that for all $T \geq T(\phi, \theta)$, there exists $R(\phi, \theta , T)$ such that for all $R \geq R(\phi, \theta ,T)$, for $k$ large enough the following approximating hold
	\begin{align*}
		\norm{U_k -\widetilde{u}_k}_{X^{\frac{1}{2}}(\{ \abs{t- t_k} \leq TN_k^{-2} \}  \cap \{\abs{t} \leq T^{-1} \})} \leq \frac{\theta}{2},
	\end{align*} 
	where
	\begin{align*}
		(\pi_{-x_k} \widetilde{u}_k) (t,x) = N_k \eta(N_k \Psi^{-1}(x) R^{-1}) u(N_k^2(t-t_k) , N_k \Psi^{-1} (x)),
	\end{align*}
	where $u$ is a solution of \eqref{cubic-R3} with scattering data $\phi^{\pm}$.
	
	For $U_k$ outside of the Euclidean window, up to subsequence, we have
	\begin{align*}
		\norm{U_k - \Pi_{t_k-t, x_k} T_{N_k} \phi^{\pm}}_{X^{\frac{1}{2}} (\{ \pm (t-t_k) \geq TN_k^{-2}\} \cap \{ \abs{t} \leq T^{-1} \})} \leq \frac{\theta}{4},
	\end{align*}
	for $k = k(\phi,\theta,T,R)$ large enough.
	
	By the extinction Lemma \ref{lem-ep-extic}, we also can take sufficiently large $T_{\theta} > T(\phi, \theta)$, such that
	\begin{align*}
		\norm{e^{it\Delta} \Pi_{t_k, x_k} T_{N_k}\phi^{\pm \infty}  }_{Z(T_{\theta}N_k^{-2} , T_{\theta}^{-1})} \leq \frac{\theta}{4} ,
	\end{align*}
	for $k$ large enough.
	
	Therefore, we can choose $R_{\theta} = R(\phi, \theta , T_{\theta})$. Then, we construct $\omega^{\theta , \pm}_k,\omega_k^{\theta}$ and $\rho_k^{\theta}$ as follows:
	\begin{enumerate}
		\item
		\begin{equation*}
			\omega_k^{\theta, \pm \infty} : = \mathbf{1}_{ \{\pm(t-t_k) \geq T_{\theta}N_k^{-2} , \abs{t} \leq T_{\theta}^{-1} \}} (\Pi_{t_k -t, x_k} T_{N_k} \phi^{\theta, \pm \infty}),
		\end{equation*}
		where
		\begin{equation*}
			\phi^{\theta , \pm} = P_{\le R_{\theta}} \phi^{\theta , \pm}.
		\end{equation*}
		
		\item
		\begin{equation*}
			\omega_k^{\theta}: = \widetilde{u}_k \cdot \mathbf{1}_{S_k^{\theta}}.
		\end{equation*}

		\item
		\begin{equation*}
			\rho_k : = \mathbf{1}_{(-T_{\theta}^{-1} , T_{\theta}^{-1})} (t)\big( U_k^{\alpha} - \omega_k^{\theta} -\omega^{\theta , +\infty} - \omega^{\theta, -\infty} \big).
		\end{equation*}
	\end{enumerate}
	
	We then verify that $\omega^{\theta , \pm}_k,\omega_k^{\theta}$ and $\rho_k^{\theta}$ satisfy $\eqref{fml-ep-ndecomp-proper}$.
	\begin{enumerate}
		\item By the definition, we have
		\begin{align*}
			\norm{\phi^{\theta, \pm \infty}}_{\dot{H}^{\frac{1}{2}} (\R^3)} \lesssim 1 , \phi^{\theta, \pm \infty} = P_{\leq R_{\theta}} \phi^{\theta , \pm \infty},
		\end{align*}
		which implies $\omega_k^{\theta , \pm \infty} = P_{\leq R_{\theta} N_{\theta}} \omega_k^{\theta, \pm \infty}$.
		
		\item By the stability Proposition \ref{prop-stability} and Lemma \ref{lem-window-profile} we can make mirror change to $\omega_k^{\theta}$ and $\omega_k^{\theta , \pm \infty}$ such that 
		\begin{align*}
			\abs{\nabla_x^m \omega_k^{\theta}} + N_k^{-2} \mathbf{1}_{S_k^{\theta}} \abs{\partial_t \nabla_x^m \omega_k^{\theta}} \leq R_{\theta} N_k^{\abs{m}+1} \mathbf{1}_{S_k^{\theta}} , \quad 0 \leq \abs{m} \leq 5 .
		\end{align*}
		
		\item By Proposition \ref{prop-ep-Ec-pro}, we obtain that
		\begin{align*}
			\norm{\rho_k^{\theta} }_{X^{\frac{1}{2}} (\{ \abs{t} < T_{\theta}^{-1}\})} \leq \frac{\theta}{2},
		\end{align*}
		and then we have
		\begin{align*}
			& \norm{\omega_k^{\theta, \pm \infty}}_{Z' ([-T_{\theta}^{-1} , T_{\theta}^{-1} ])} + \norm{\rho_k^{\theta}}_{X^{\frac{1}{2}} ([-T_{\theta}^{-1} , T_{\theta}^{-1} ])} \leq \theta,\\
			& \norm{\omega_k^{\theta, \pm \infty}}_{X^{\frac{1}{2}}([-T_{\theta}^{-1} , T_{\theta}^{-1}])} + \norm{\omega_k^{\theta}}_{X^{\frac{1}{2}} ([-T_{\theta}^{-1} , T_{\theta}^{-1}])} \lesssim 1.
		\end{align*}
	\end{enumerate}

	So far, we have finished the proof of this Corollary.
\end{proof}

\subsection{Profile decomposition}
In this section, we construct the Euclidean profiles of bounded sequence of functions on $\T^3$ and express the sequence(up to a subsequence) as the almost orthogonal sum of these profiles.

In the previous, we have defined the renormalized Euclidean frames. Now, we define the Euclidean frames by dropping the assumption that either $t_k = 0$ or $\lim\limits_{t \to \pm\infty} N_k^2 t_k = +\infty $.

\begin{definition}[Euclidean frames]
	\begin{enumerate}
		\item We define a Euclidean frame to be a sequence 
		\begin{equation*}
			\mathcal{F}_e := \{ (N_k, t_k, x_k)_{k\le1} : N_k\geq 1, N_k\to +\infty, t_k\to 0, x_k\in\T^3 \}. 
		\end{equation*}
		We say that two frames, $(N_k, t_k, x_k)_k$ and $(M_k, s_k, y_k)_k$ are orthogonal if
		\begin{align*}
			\lim_{k\to+\infty} \parenthese{ \ln \abs{\frac{N_k}{M_k}}+ N_k^2 \abs{t_k-s_k} + N_k\abs{x_k-y_k} } =\infty.
		\end{align*}
		Otherwise, we say that frames $(N_k, t_k, x_k)_k$ and $(M_k, s_k, y_k)_k$ are equivalent.

		\item If $\mathcal{O} =(N_k, t_k, x_k)_k$ is a Euclidean frame and if  $\phi\in\dot{H}^{\frac{1}{2}}(\R^3)$, we define the Euclidean profile associated to  $(\phi, \mathcal{O})$ as the sequence $\widetilde{\phi}_{\mathcal{O}_k}$:
		\begin{align*}
			\widetilde{\phi}_{\mathcal{O}_k} := \Pi_{t_k, x_k} (T_{N_k}\phi).
		\end{align*}
	\end{enumerate}
\end{definition}

In the following Proposition, we collect basic properties associated to orthogonal and equivalent Euclidean frames.
\begin{proposition}[Properties on Euclidean frames]\label{prop-ep-Eucpro-proper}
	\begin{enumerate}
		\item
		If $\mathcal{O}$ and $\mathcal{O}'$ are equivalent Euclidean frames, then there
		exists an isometry $T: \dot{H}^{\frac{1}{2}}(\R^3) \to \dot{H}^{\frac{1}{2}}(\R^3)$ such that for any profile $\widetilde{\phi}_{\mathcal{O}'_k}$, up to a subsequence there is
		\begin{align*}
			\limsup_{k\to\infty} \norm{\widetilde{T \phi}_{\mathcal{O}_k} - \widetilde{\phi}_{\mathcal{O}'_k}}_{H_x^{\frac{1}{2}}(\T^3)} = 0.
		\end{align*}
		
		\item
		If $\mathcal{O}$ and $\mathcal{O}'$ are orthogonal Euclidean frames and $\widetilde{\phi}_{\mathcal{O}_k}$, $\widetilde{\phi}_{\mathcal{O}'_k}$ are corresponding profiles, then, up to a subsequence:
		\begin{align}\label{fml-ep-H1/2-Otho}
			\lim_{k\to\infty} \inner{ \widetilde{\phi}_{\mathcal{O}_k},  \widetilde{\phi}_{\mathcal{O}'_k}}_{H^{\frac{1}{2}}\times H^{\frac{1}{2}}(\T^3)} = 0;\\
			\lim_{k\to\infty} \|\widetilde{\phi}_{\mathcal{O}_k}  \widetilde{\phi}_{\mathcal{O}'_k}\|_{L^{\frac{3}{2}}(\T^3)} = 0.
		\end{align}
		
	\end{enumerate}
\end{proposition}

The proof is rather standard, so we omit the details, see Lemma 5.4 in \cite{IPT3} for the proof.

\begin{definition}[Absence from a frame]
	We say that a sequence of functions $\{ f_k\}_k \subset H^{\frac{1}{2}} (\T^3)$ is absent from a frame $ \mathcal{O}$ if, up to a subsequence, for every profile $\psi_{\mathcal{O}_k}$ associated to $\mathcal{O}$, 
	\begin{gather*}
		\inner{ f_k , \widetilde{\psi}_{\mathcal{O}_k} }_{H^{\frac{1}{2}}(\T^3)\times H^{\frac{1}{2}}(\T^3)} \to 0,\\
		\langle f_k , \widetilde{\psi}_{\mathcal{O}_k} \rangle_{L^2(\T^3)\times L^2(\T^3)} \to 0.
	\end{gather*}
	as $k \to \infty$.
\end{definition}

\subsection{Linear profile decomposition }
The next proposition shows the profile decomposition for a bounded  sequence $\{f_n\}\subset H^\frac{1}{2}(\T^3)$.

To establish the linear profile decomposition, we need the refined Strichartz estimates in Besov space and the inverse Strichartz estimates following the argument in Keraani\cite{Keraani2001}.
\begin{lemma}\label{lem-ep-Ref-Stri}
	Let $f\in H^{\frac{1}{2}}(\T^3)$, $I\subset [0,1]$ and $a \in (0,1)$, then 
	\begin{align*}
		\norm{e^{it\Delta} f}_{Z(I)} \lesssim \norm{f}_{H_x^\frac{1}{2}(\T^3)}^{a} \sup_{N\in 2^{\Z}} \big( N^{-1} \norm{P_N e^{it\Delta}f}_{L_{t,x}^{\infty}(I\times\T^3)} \big)^{1-a}.
	\end{align*}
\end{lemma}
\begin{proof}
	By the definition of $Z$-norm, we only need consider the case
	\begin{align*}
		\norm{e^{it\Delta} f}_{Z(I)} = \parenthese{\sum_N N^{5 - p} \norm{P_N e^{it\Delta} f}_{L^p_{t,x}}^p }^{\frac{1}{p}} = \norm{N^{\frac{1}{p}} \norm{P_N e^{it\Delta}f}_{L^p_{t,x}} }_{l^p_N}.
	\end{align*}
	By H\"{o}lder inequality, for any $0 < a < 1$ and $\frac{10}{3} < ap < \infty$, we get
	\begin{align}\label{fml-ep-prof-lem-Str-1}
		\norm{N^{\frac{5}{p} - 1} \norm{P_N e^{it\Delta}f}_{L_{t,x}^p} }_{l_N^p} & \lesssim  \norm{  \parenthese{N^{\frac{5}{ap} - 1} \norm{P_Ne^{it\Delta}f}_{L^{ap}_{t,x}}}^{a} \parenthese{N^{-1}   \norm{P_N e^{it\Delta} f}_{L_{t,x}^{\infty}}}^{1-a}}_{l_N^p} \nonumber\\
		&  \lesssim  \sup_{N} \parenthese{ N^{-1} \norm{P_N e^{it\Delta} f }_{L_{t,x}^{\infty}} }^{1-a}  \parenthese{ \sum_N N^{5 - ap} \norm{P_N e^{it\Delta} f }_{L_{t,x}^{ap}}^{ap} }^{\frac{1}{p}}.
	\end{align}
	Using Proposition \ref{prop-pre-Stri} and Proposition \ref{prop-pre-Embed}, we obtain
	\begin{equation}\label{fml-ep-prof-lem-Str-2}
		N^{5 - ap} \norm{P_N e^{it\Delta} f }_{L_{t,x}^{ap}}^{ap} \lesssim \|P_N f\|_{H_x^{\frac{1}{2}}}^{ap}.
	\end{equation}
	Combining \eqref{fml-ep-prof-lem-Str-1} with \eqref{fml-ep-prof-lem-Str-2}, we have finished the proof of this lemma. 
\end{proof}
Now, we show the following inverse Strichartz estimates.
\begin{lemma}\label{lem-ep-decomp}
Let $\{f_k\}_{k\ge1}$ be an uniformly bounded sequence of functions in $H^{\frac{1}{2}}(\T^3)$ and up to subsequence $f \rightharpoonup g \in H^{\frac{1}{2}}(\T^3)$. Moreover, we assume that there exists a sequence of interval $0 \in I_k = (-T_k , T^k)$ such that $|I_k| \to 0$ as $k \to 0$. Then, for fixed $\delta > 0$, there exist $J \lesssim \delta^{-2}$ and profiles $\widetilde{\phi}^{\alpha}_{O_k^{\alpha}}$ associated to pairwise orthogonal Euclidean frame $\mathcal{O}^{\alpha},\alpha = 1,2,\cdots,J$ such that
\begin{equation*}
	f_k = g + \sum_{1\le \alpha \le J} \widetilde{\phi}^{\alpha}_{O_k^{\alpha}} + R_{k}^J.
\end{equation*}
The $R_k^{J}$ is the remainder in the sense that
\begin{equation}\label{fml-ep-lem-remainder}
	\limsup_{k \to \infty} \|e^{it\Delta} R_k\|_{Z(I_k)} \le \delta.
\end{equation}
Furthermore, we have the following orthogonal relations
\begin{gather*}
	\|f_k\|_{L^2(\T^3)}^2 = \|g\|_{L^2(\T^3)}^2 + \|R_k\|_{L^2(\T^3)}^2 + o_k(1),\\
	\|f_k\|_{\dot{H}^{\frac{1}{2}}(\T^3)}^2 = \|g\|_{\dot{H}^{\frac{1}{2}}(\T^3)}^2 + \sum_{1 \le \alpha \le J}\|\psi^{\alpha}\|_{\dot{H}^{\frac{1}{2}}(\R^3)}^2 + \|R_k\|_{\dot{H}^{\frac{1}{2}}(\T^3)}^2 + o_k(1).
s\end{gather*} 
\end{lemma}
\begin{proof}
First,	we prove  \eqref{fml-ep-lem-remainder}. We define  $\Lambda\big( (f_k)_k \big) $ as follows
	\begin{equation*}
		\Lambda\big( (f_k)_k \big) := \limsup_{k \to +\infty} \sup\limits_{N\ge1, t \in I_k, x\in \T^3} N^{-1} \big| (e^{it\Delta} P_N f_k) (x) \big|.
	\end{equation*}
	By Lemma \ref{lem-ep-Ref-Stri}, it is sufficient to prove
	\begin{equation*}
		\Lambda\big(R_k\big) \le \delta.
	\end{equation*}
	Notice that $H^{\frac{1}{2}}(\T^3)$ is a separable Hilbert space, we have $\inner{ f_k - g , g } \to 0$ as $k \to 0$. Therefore, we get
	\begin{equation*}
		\|f_k\|_{H^{\frac{1}{2}}(\T^3)}^2 = \|f_k - g\|^2_{H^{\frac{1}{2}}(\T^3)} + \|g\|_{H^{\frac{1}{2}}(\T^3)}^2.
	\end{equation*}
	Similarly, we can also find
	\begin{equation*}
		\|f_k\|_{H^{\frac{1}{2}}(\T^3)}^2 = \|f_k - g\|^2_{H^{\frac{1}{2}}(\T^3)} + \|g\|_{H^{\frac{1}{2}}(\T^3)}^2.
	\end{equation*}
	Hence, without of generality, we assume that $g = 0$.
	
	If $\Lambda\big( (f_k)_k \big) < \delta$, then \eqref{fml-ep-lem-remainder} follows immediately. Otherwise, if $\Lambda\big( (f_k)_k \big) \ge \delta$, we claim that there exists an Euclidean frame $\mathcal{O}$ and an associated profile $\big( \widetilde{\psi}_{\mathcal{O}_k} \big)_k$ such that
	\begin{equation}\label{fml-ep-lem-prof-bound}
		\limsup_{k \to +\infty} \| \widetilde{\psi}_{\mathcal{O}_k} \|_{H^{\frac{1}{2}}(\T^3)} \lesssim \|\psi\|_{H^{\frac{1}{2}}(\R^3)},
	\end{equation}
	and 
	\begin{equation}\label{fml-ep-lem-prof-absent}
		\limsup_{k \to +\infty} \big| \langle f_k , \widetilde{\psi}_{\mathcal{O}_k} \rangle_{H^\frac{1}{2}(\T^3)} \big| \ge \frac{\delta}{2}.
	\end{equation}
	Moreover, if $\big(f_k\big)_k$ is absent from a family of frames $\big(\mathcal{O}^{\alpha}\big)_{\alpha}$, then $\mathcal{O}$ is orthogonal to all $\mathcal{O}_{\alpha}$.
	
	Now, we prove the claim. By the definition and the fact that $T_N$ is a bounded operator, we get \eqref{fml-ep-lem-prof-bound} immediately. Hence, we only need to prove \eqref{fml-ep-lem-prof-absent}. By the definition of $\Lambda$, up to a subsequence, there exists $(N_k,t_k,x_k)_k$ with $(N_k,t_k,x_k) \in [1,\infty)\times I_k \times \T^3$ such that
	\begin{align*}
		\frac{\delta}{2} & \le N_k^{-1}\big| (e^{it_k\Delta})P_{N_k}f(x_k) \big| = \bigg| \langle N_k^{-1}(e^{it_k\Delta})P_{N_k}f , \delta_{x_k} \rangle \bigg|\\
		& \le \bigg| \langle f , N_k^{-1}(e^{-it_k\Delta})P_{N_k}\delta_{x_k} \rangle_{H^{\frac{1}{2}}\times H^{-\frac{1}{2}}} \bigg|.
	\end{align*}
	We also claim that $N_k \to \infty$. Otherwise, we can find $1\le N_{\infty} < \infty$ and $t_{\infty} \in \T^3$ such that $N_k \to N_{\infty}$ and $t_k \to t_{\infty}$ up to a subsequence, respectively. Let
	\begin{equation*}
		\overline{\psi} := (1 - \Delta)^{-\frac{1}{2}}N_{\infty}^{-1}P_{N_{\infty}}\delta_{x_{\infty}}.
	\end{equation*}
	Notice that $N^{-1}e^{-it\Delta}P_{N}\delta_{x}$ is continuous with respect to $(N,t,x)$. Hence, we can find $\overline{\psi} \in H^{\frac{1}{2}}$ such that
	\begin{equation*}
		\langle f_k , \overline{\psi} \rangle \ge \frac{\delta}{4}
	\end{equation*}
	for all $k$ sufficiently large, which contradict with the fact that $f_k \rightharpoonup 0$.
	
	We define 
	\begin{equation*}
		\psi := \mathcal{F}^{-1}_{\R^3} \big( |\xi|^{-2} (\eta^3(\xi) - \eta^3(2\xi)) e^{ix_0 \cdot \xi } \big),
	\end{equation*}
	where $\eta$ is a bump function. Then, we are devoted to verify the profile $\widetilde{\psi}_{\mathcal{O}_k}$ associated with the Euclidean frame $\mathcal{O} = (N_k,t_k,x_k)_k$ satisfies assertion \eqref{fml-ep-lem-prof-absent}. Using the fact that $N_K \to \infty$ and Paserval's identity, we obtain
	\begin{equation*}
		\lim_{k \to \infty} \big\| (1 - \Delta) T_{N_k}\psi - N_k^{-1}P_{N_k}\delta_{x_0} \big\|_{L^{\frac{3}{2}}} = 0,  
	\end{equation*}
	and by the Sobolev embedding we also get
	\begin{equation*}
		\lim_{k \to \infty} \big\| (1 - \Delta) T_{N_k}\psi - N_k^{-1}P_{N_k}\delta_{x_0} \big\|_{H^{-\frac{1}{2}}} = 0.
	\end{equation*}
	We conclude that 
	\begin{equation*}
		\frac{\delta}{2} \lesssim \abs{\inner{ f_k , N_k^{-\frac{1}{2}} e^{it_k\Delta} P_{N_k} \delta_{x_k}}_{H^{\frac{1}{2}} \times H^{-\frac{1}{2}} }} \lesssim \abs{ \inner{f_k, \widetilde{\psi}_{\mathcal{O}_k}}_{H^{\frac{1}{2}} \times H^{\frac{1}{2}}}}.
	\end{equation*}
	Therefore, we have construct a Euclidean frame $\mathcal{O}$ satisfies \eqref{fml-ep-lem-prof-bound} and \eqref{fml-ep-lem-prof-absent}. Assuming that $(f_k)_k$ is absent from a family of frames $(\mathcal{O}^{\alpha})_{\alpha}$, then for every $\alpha$ and profile $\big( \widetilde{\psi}_{\mathcal{O}_{k}^{\alpha}} \big)_{k}$
	\begin{equation*}
		\langle f_k , \widetilde{\psi}_{\mathcal{O}_{k}^{\alpha}} \rangle_{H^{\frac{1}{2}}(\T^3)} \to 0, \quad \text{as} \quad k \to \infty.
	\end{equation*}
	Assuming that there exists $\mathcal{O}^{\beta} \in (\mathcal{O}^{\alpha})_{\alpha}$ such that $\mathcal{O}^{\beta}$ is equivalent to $\mathcal{O}$. On one hand, by Proposition \ref{prop-ep-Eucpro-proper}, there exists isometry $T$ such that
	\begin{equation*}
		\limsup_{k \to \infty} \| \widetilde{\psi}_{\mathcal{O}_{k}} - \widetilde{T \psi}_{\mathcal{O}_{k}^{\beta}} \|_{H^{\frac{1}{2}}(\T^3)} = 0.
	\end{equation*}
	On the other hand, from \eqref{fml-ep-lem-prof-absent}
	\begin{align}\label{fml-ep-lem-contr}
		\frac{\delta}{2} & \le \limsup_{k \to \infty} \big| \langle f_k , \widetilde{\psi}_{\mathcal{O}_{k}} \rangle \big| \nonumber \\
		& \le \limsup_{k \to \infty}\big( \|f_{k}\|_{H^{\frac{1}{2}}(\T^3)}\| \widetilde{\psi}_{\mathcal{O}_{k}} - \widetilde{T \psi}_{\mathcal{O}_{k}^{\beta}} \|_{H^{\frac{1}{2}}(\T^3)} + \big| \langle f_k , \widetilde{S \psi}_{\mathcal{O}_{k}^{\beta}} \rangle \big| \big).
	\end{align}
	Since $(f_k)_k$ is absent from $\mathcal{O}^{\beta}$, the right hand side of \eqref{fml-ep-lem-contr} tend to zero while $k \to \infty$, which is a contradiction. Therefore, $\mathcal{O}$ is orthogonal with $\big( \mathcal{O}^{\alpha} \big)_{\alpha}$.
	
	For $R \ge 1$ and $k \ge k_0$, we define $\psi^{R}_k : \R^3 \to \C$ as follows
	\begin{equation*}
		\psi^R_k(y) := N_k^{-1} \eta^3(R^{-1}y) ( \Pi_{-t_k,-x_k} f_k ) \big( \Psi(N_k^{-1}y) \big).
	\end{equation*} 
	We can easily check that
	\begin{equation*}
		\| \psi^R_k \|_{H^{\frac{1}{2}}(\R^3)} \lesssim \|f_k\|_{H^{\frac{1}{2}}(\T^3)} \le A < \infty,
	\end{equation*}
	uniformly on $R$. Hence, we have $\psi^R,\psi \in H^{\frac{1}{2}}(\R^3)$ such that $\psi^R_k \rightharpoonup \psi^R \rightharpoonup \psi$ up to a subsequence. By the uniqueness of weak limit, we can find
	\begin{equation*}
		\psi^R(x) = \eta^3(R^{-1}x) \psi(x).
	\end{equation*}
	For any $\gamma \in C^{\infty}_0(\R^3)$ which is supported in $B(0 , R/2)$, then for $k$ large enough, we have
	\begin{align}\label{fml-ep-fk-psi-vanish}
		\langle f_k - \widetilde{\psi}_{\mathcal{O}_k} , \widetilde{\gamma}_{\mathcal{O}_k} \rangle_{H^{\frac{1}{2}}\times H^{\frac{1}{2}}(\T^3)} & = \langle \Pi_{-t_k,-x_k}f_k , T_{N_k}\gamma \rangle_{H^{\frac{1}{2}}\times H^{\frac{1}{2}}(\T^3)} - \langle \psi , \gamma \rangle_{H^{\frac{1}{2}}\times H^{\frac{1}{2}}(\R^3)}\nonumber\\
		& = \langle \psi^R , \gamma \rangle_{H^{\frac{1}{2}}\times H^{\frac{1}{2}}(\R^3)} - \langle \psi , \gamma \rangle_{H^{\frac{1}{2}}\times H^{\frac{1}{2}}(\R^3)} + o_k(1)\nonumber\\
		& = o_k(1).
	\end{align}
	By density, this is also hold for $\gamma \in H^{\frac{1}{2}}(\R^3)$. Thanks to \eqref{fml-ep-lem-prof-bound}, we also obtain
	\begin{equation*}
		\|\psi\|_{H^{\frac{1}{2}}(\R^3)} \gtrsim \delta.
	\end{equation*}
	Combining \eqref{fml-ep-fk-psi-vanish} with Proposition \ref{prop-ep-Eucpro-proper}, we get
	\begin{equation*}
		\|f_k - \widetilde{\psi}_{\mathcal{O}_k}\|_{L^2(\T^3)}^2 = \|f_k\|^2_{L^2(\T^3)} + o_k(1),
	\end{equation*}
	and
	\begin{equation*}
		\|f_k - \widetilde{\psi}_{\mathcal{O}_k}\|_{\dot{H}^{\frac{1}{2}}(\T^3)}^2 = \|f_k\|^2_{\dot{H}^{\frac{1}{2}}(\T^3)} + \|\psi\|^2_{\dot{H}^{\frac{1}{2}}(\T^3)} + o_k(1).
	\end{equation*}
	
	We set $f_k^0 := f_k$ and $\big(f_k^0\big)_k := \big(f_k\big)_k$. From the previous procedures, for $\alpha \ge 0$ and $\Lambda\big( (f_k^{\alpha})_k \big)$, we can always construct Euclidean frame $\mathcal{O}_{\alpha}$ and associated profile $\widetilde{\psi}_{\mathcal{O}_k}^{\alpha}$. Then, we define
	\begin{equation*}
		f_k^{\alpha + 1} := f_k^{\alpha} - \widetilde{\psi}^{\alpha}_{\mathcal{O}_{\alpha}}, \quad k \ge 1.
	\end{equation*}
	We have proved that $\mathcal{O}^{\alpha}$ is orthogonal with $\mathcal{O}^{\beta}$ with $\beta < \alpha$. Moreover, $f_k^{\alpha + 1}$ is absent from $\mathcal{O}^{\alpha}$. By the equivalence, $f_k^{\alpha + 1}$ is absent from $\mathcal{O}_{\beta}$ with all $\beta \le \alpha$. So far, for every $\alpha + 1 \ge 1$, we can find family of frames $\big(\mathcal{O}_{\beta}\big)_{\beta \le \alpha}$ associated profiles $\big(\widetilde{\psi}_{\mathcal{O}_{\beta}}\big)_{\beta \le \alpha}$
	\begin{equation*}
		\|f_k\|_{L^2}^2 = \sum_{\beta \le \alpha}\|\widetilde{\psi}_{\mathcal{O}^{\beta}_k}\|_{L^2}^2 + \|f_k^{\alpha + 1}\|_{L^2} + o_k(1),
	\end{equation*}
	and
	\begin{equation*}
		\|f_k\|_{\dot{H}^{\frac{1}{2}}}^2 = \sum_{\beta \le \alpha}\|\widetilde{\psi}_{\mathcal{O}^{\beta}_k}\|_{\dot{H}^{\frac{1}{2}}}^2 + \|f_k^{\alpha + 1}\|_{\dot{H}^{\frac{1}{2}}} + o_k(1).
	\end{equation*}
	Denoting that $R_k := f_k - \sum_{\beta \le \alpha} \widetilde{\psi}_{\mathcal{O}^{\beta}_k}$, by the orthogonality we get
	\begin{equation*}
		\|R_k\|_{H^{\frac{1}{2}}}^2 = \bigg| \|f_k\|_{H^{\frac{1}{2}}}^2 - \sum_{\beta \le \alpha} \|\widetilde{\psi}_{\mathcal{O}^{\beta}_k}\|_{H^{\frac{1}{2}}}^2 \bigg| + o_k(1) \lesssim |A - \alpha \delta^2|. 
	\end{equation*}
	Taking $\alpha \sim \delta^{-2}$, we have $\Lambda(R_k) \le \delta$.

\end{proof}
Applying Lemma \ref{lem-ep-Ref-Stri} and Lemma \ref{lem-ep-decomp}, we have the following linear profile decomposition.
\begin{proposition}[Linear profile decomposition]\label{prop-ep-prop-decomp}
	Let $\{ f_k\}_{k}$ be a sequence of bounded functions in $H^{\frac{1}{2}}(\T^3)$ and satisfy
	\begin{align*}
		\limsup_{k\to +\infty} \norm{f_k}_{H^{\frac{1}{2}}(\T^3)} \leq A < +\infty,
	\end{align*}
	and $f_k \rightharpoonup g \in H^{\frac{1}{2}}(\T^3)$ up to a subsequence. We further let a sequence of intervals $I_k=(-T_k, T^k)$ around the origin such that $\abs{I_k} \to 0$ as $k\to\infty$. Then, there exists $J^*\in \mathbb{N}$, and a sequence of profile $ \widetilde{\psi}_{\mathcal{O}_k^{\alpha}}^{\alpha}$ associated to pairwise orthogonal Euclidean frames $(\mathcal{O}^{\alpha})_{\alpha}$ and $\psi^{\alpha} \in {H}^{\frac{1}{2}}(\R^3)$ such that extracting a subsequence, for every $0 \leq J\leq J^*$, we have
	\begin{align}\label{5.1}
		f_k = g + \sum_{1\leq \alpha \leq J} \widetilde{\psi}_{\mathcal{O}_k^{\alpha}}^{\alpha}  + R_k^J
	\end{align}
	where $R^J_k$ is the remainder in the sense that
	\begin{align}\label{5.2}
		\limsup_{J \to J^*}\limsup_{k\to\infty}\norm{e^{it\Delta}R_k^J}_{Z(I_k)} = 0.
	\end{align}
	Besides, we also have the following orthogonality relations:
	\begin{equation}\label{fml-ep-prof-Otho}
		\begin{aligned}
			& \norm{f_k}_{L^2}^2 = \norm{g}_{L^2}^2 + \norm{R_k^J}_{L^2}^2 + o_k (1),\\
			& \norm{f_k}_{\dot{H}^{\frac{1}{2}}}^2  = \norm{g}_{\dot{H}^{\frac{1}{2}}}^2 + \sum_{\alpha \leq J} \norm{\psi^{\alpha}}_{\dot{H}^{\frac{1}{2}}}^2 + \norm{R_k^J}_{\dot{H}^{\frac{1}{2}}}^2 + o_k(1) \\
			& \limsup_{J \to J^*}\limsup_{k\to \infty} \left| \|f_k\|_{L^3}^3 - \|g\|_{L^3}^3 - \|\widetilde{\psi}_k^{\alpha}\|_{L^3}^3 \right| = 0.
		\end{aligned}
	\end{equation}
	
\end{proposition}

As an application, we have the following orthogonality result
\begin{proposition}\label{prop-ep-nonlin-ortho}
	Let $U^{\alpha}_k,U^{\beta}_k$ and $U^{\gamma}_k$ denote the maximal lifespan $0 \in I_k$ solutions of \eqref{cubic-intro} with initial data $U_k^{\delta}(0) = \widetilde{\phi^{\delta}_{ \mathcal{O}^{\delta}_k }}$ for $\delta = \alpha,\beta,\gamma$. The Euclidean frames $\mathcal{O}^{\delta}_k,\delta = \alpha,\beta,\gamma$ are orthogonal with each other and $|I_k| \to 0$ as $k \to \infty$. We also denote $G$ as the maximal lifespan $I_0$ solution of \eqref{cubic-intro} with initial data $G(0) = g$. Then
	\begin{equation*}
		\limsup_{k \to \infty} \| V^{\alpha}_k V^{\beta}_k V^{\gamma}_k \|_{N(I_k)} = 0,
	\end{equation*}
	where $V^{\delta} \in \{ U^{\alpha}, \overline{U}^{\alpha} , U^{\beta}, \overline{U}^{\beta} , U^{\gamma} , \overline{U}^{\gamma} , G , \overline{G} \}$ with $\delta = \alpha,\beta,\gamma$.
\end{proposition}
The proof of this proposition is similar with that of \cite{Yu2021,Z2}, so we omit it.

\section{Proof of the main Theorem}
	In this section, we prove Theorem \ref{thm} by using the concentration-compactness/rigidity argument.
	
	We define
	\begin{equation*}
		\Lambda(L,\tau) := \sup\{ \|u\|_{Z(I)} : \|u\|_{H^{\frac{1}{2}}(\T^3)} < L, |I| < \tau \},
	\end{equation*}
	where $u$ is the strong solution to \ref{cubic-intro} with initial data $u_0$ in the interval $I$. It is not hard to find $\Lambda(L,\tau)$ is increasing function with $L$ and $\tau$. $\Lambda(L,\tau)$ have the sublinearity in $\tau$, that is for every $\tau,\sigma > 0, \Lambda(L,\tau + \sigma) \le \Lambda(L,\tau) + \Lambda(L,\sigma)$. Then, we define 
	\begin{equation*}
		\Lambda_{*}(L) := \lim_{\tau \to 0} \Lambda(L,\tau)
	\end{equation*}
	and the sublinear property guarantees for all fixed $\tau > 0$
	\begin{equation*}
		\Lambda(L,\tau) < \infty, \Longleftrightarrow \Lambda_{*}(L) < \infty.
	\end{equation*}
	Finally, we define the maximal energy such that $\Lambda_{*}(L)$ is finite:
	\begin{equation}\label{fml-thm-Emax}
		E_{max} := sup\{ L \in [0,+\infty]: \Lambda_{*}(L) < \infty \}.
	\end{equation}
	Note that the small initial data global well-posedness of \ref{cubic-intro} insures $E_{max} > 0$. To obtain Theorem \ref{thm-ep-GWP}, it is sufficient to show
	\begin{theorem}\label{thm-finite-energy}
		Let $E_{max}$ be as in \eqref{fml-thm-Emax}, then $E_{max} = \infty$.
	\end{theorem}
	
	We argue by contradiction, assume that $E_{max} < \infty$ and define
	\begin{equation*}
		L(\phi) := \sup_{t \in t_k}\{ \|u_{\phi}(t)\|^2_{H^{\frac{1}{2}}(\T^3)} \},
	\end{equation*}
	where $u_{\phi}$ denote the solution of \eqref{cubic-intro} with initial data $u(0) = \phi$.
	There exist a sequence of strong solution $\{u_k\}_k$ to \ref{cubic-intro} such that 
	\begin{equation}\label{fml-thm-Znorm-inf}
		L(u_k) \to E_{max}, \text{and} \quad \|u_k\|_{Z(I_k)} \to \infty,
	\end{equation}
	where $0 \in I_k$ and $|I_k| \to 0$. In this setting, $(u_k(0))_k$ is uniformly bounded, hence there exist unique $g \in H^{\frac{1}{2}}(\T^3)$ such that $u_k(0) \rightharpoonup g$. By proposition \ref{prop-ep-prop-decomp}, we can find a family of frames $\big(\mathcal{O}^{\alpha}\big)_{\alpha}$ and associated profiles $\big( \widetilde{\psi}^{\alpha}_{\mathcal{O}_{\alpha}} \big)_k$ such that, up to a subsequence, $u_k(0)$ has the decomposition
	\begin{equation}
		u_k(0) = g + \sum_{\alpha = 1}^J \widetilde{\psi}^{\alpha}_{\mathcal{O}_{\alpha}} + R^J_k.
	\end{equation}
	The remainder $R^J_k$ is small in the sense that
	\begin{equation*}
		\limsup_{J \to \infty}\limsup_{k \to \infty} \| e^{it\Delta} R_k^J \|_{Z(I_k)} = 0.
	\end{equation*}
	However, to apply Proposition \ref{prop-ep-prop-decomp}, we need replace frame $\mathcal{O}^{\alpha}$ by renormalized Euclidean frame $\widetilde{\mathcal{O}}^{\alpha}$. Considering $\mathcal{O}_{\alpha} = \{N_k, t_k, x_k\}_k$ is not a renormalized Euclidean frame, after passing to a subsequence, we have $|t_k|N_k^2 \to C$ for constant $0 \le C < \infty$ as $k \to \infty$. We define frame $\widetilde{\mathcal{O}}^{\alpha} := \{N_k, 0, x_k\}_k$, which is equivalent to $\mathcal{O}^{\alpha}$. By Proposition \ref{prop-ep-Eucpro-proper}, there exists a profile $\big( \widetilde{T_{\alpha} \psi^{\alpha}}_{\widetilde{\mathcal{O}}^{\alpha}_k} \big)_k$($T_{\alpha}$ denotes the isometry as in Proposition \ref{prop-ep-Eucpro-proper}), up to a subsequence, such that
	\begin{equation*}
		\lim_{k \to \infty} \| \widetilde{T_{\alpha} \psi^{\alpha}}_{\widetilde{\mathcal{O}}^{\alpha}_k} - \widetilde{\psi}^{\alpha}_{\mathcal{O}^{\alpha}_k} \|_{H^{\frac{1}{2}}(\T^3)} = 0.
	\end{equation*}
	Let $\widetilde{u}_k(0) := g + \sum_{\alpha = 1}^J \widetilde{T_{\alpha} \psi^{\alpha}}_{\widetilde{\mathcal{O}}^{\alpha}_k} + R^J_k$ and $\widetilde{u}_k$ be the solution to \eqref{cubic-intro} with initial data $\widetilde{u}_k(0)$. Applying the stability result in Proposition \ref{prop-stability}, we guarantee existence of $\widetilde{u}_k$ for sufficiently large $k$.  
	
	Supposing that $\|\widetilde{u}_k(0)\|_{Z(I_k)}$ is uniformly bounded, we can find $\|\widetilde{u}_k\|_{X^{\frac{1}{2}}(I_k)}$ is uniformly bounded, which implies
	\begin{equation*}
		\|\widetilde{u}_k\|_{Z(I_k)} + \|\widetilde{u}_k\|_{L^{\infty}(I_k,H^{\frac{1}{2}}(\T^3))} < \infty
	\end{equation*}
	uniformly in $k$. We also notice that
	\begin{equation*}
		\|u_k\|_{Z(I_k)} \le \|u_k - \widetilde{u}_k\|_{X^{\frac{1}{2}}(I_k)} + \|\widetilde{u}_k\|_{Z(I_k)},
	\end{equation*}
	combine it with Proposition \ref{prop-stability}, we get $\|u_k\|_{Z(I_k)}$ is uniformly bounded in $k$. Consequently, we only need to prove the case when $\mathcal{O}^{\alpha}$ are renormalized frames.
	
	Denoting that
	\begin{equation*}
		\mathcal{L}(\alpha) := \lim_{k \to \infty} L( \widetilde{\psi}^{\alpha}_{\mathcal{O}^{\alpha}_k} ),
	\end{equation*}
	and thanks to the orthogonality relation \eqref{fml-ep-ndecomp-proper} obtained in Proposition \ref{prop-ep-prop-decomp}, we divide the norm $\|\cdot\|_{L_t^{\infty}H_x^{\frac{1}{2}}(\T^3)}$ into three parts:
	\begin{equation*}
		\lim_{J \to \infty} \bigg( \sum_{\alpha = 1}^J \mathcal{L}(\alpha) + \lim_{k \to \infty} L(R^J_k) \bigg) + \mathcal{L}(g) \le E_{max}.
	\end{equation*}
	We call that $g$ is the scale-one profile and $\widetilde{\psi}^{\alpha}_{\mathcal{O}^{\alpha}_k}$ are Euclidean profile. The proof of Theorem \ref{thm-finite-energy} can be deduce to the following three cases
	\begin{enumerate}
		\item No Euclidean profiles,\\
		\item Only one Euclidean profile, while scale-one profile vanishes,\\
		\item other cases.
	\end{enumerate} 
	\begin{proof}[Proof of Theorem \ref{thm-finite-energy}]
		\textbf{Case (1): No Euclidean profiles:} Notice that $|I_k| \to 0$ as $k \to \infty$, then for any $\varepsilon > 0$ there exist $\eta > 0$ such that for sufficiently large $k$
		\begin{equation*}
			\|e^{it\Delta}u_k(0)\|_{Z(I_k)} \le \|e^{it\Delta}g\|_{Z(-\eta,\eta)} + \varepsilon \le 2\varepsilon.
		\end{equation*} 
		By proposition \ref{prop-lwp-LWP}, we have
		\begin{equation*}
			\|u_k\|_{Z(I_k)} \le 4 \varepsilon,
		\end{equation*}
		which contradicts with \eqref{fml-thm-Znorm-inf}.
		
		\textbf{Case (2): Only one Euclidean profile, while scale-one profile vanishes.} In this case, $g = 0$ identically, meanwhile, there exists only one profile $\widetilde{\psi}^1_P{\mathcal{O}^1_k}$ such that $\mathcal{L}(1) = E_{max}$. Let $U_k$ be the solution to \eqref{cubic-intro} with initial data $U^1_k(0) = \widetilde{\psi}^1_P{\mathcal{O}^1_k}$. From Proposition \ref{prop-lwp-LWP}, there exist $\tau > 0$ and $U^1_k \in X^{\frac{1}{2}}(-\tau,\tau)$ satisfies 
		\begin{equation*}
			\|U^1_k\|_{X^{\frac{1}{2}}(-\tau,\tau)} \lesssim 1.
		\end{equation*}
		Note that we can find $k$ large enough, such that $I_k \subset (-\tau,\tau)$. Hence, by the embedding relation, we get
		\begin{equation*}
			\|U^1_k\|_{Z(I_k)} + \|U^1_k\|_{L^{\infty}( I_k , H^{\frac{1}{2}}(\T^3) )} \lesssim 1,
		\end{equation*}
		for sufficiently large $k$. Combining this with 
		\begin{equation*}
			\limsup_{k \to \infty} \| u_k(0) - U^1_k(0) \|_{H^{\frac{1}{2}}(\T^3)} = 0,
		\end{equation*}
		allows us to use the stability result obtained in Proposition \ref{prop-stability}. Finally, we have
		\begin{equation*}
			\|u_k\|_{Z(I_k)} \lesssim 1,
		\end{equation*}
		for $k$ large enough, which contradict with \eqref{fml-thm-Znorm-inf}.
		
		\textbf{Case (3): other cases:} In this case, observe that $L(g) < E_{max}$ and $\mathcal{L}(\alpha) < E_{max}$ for every $1 \le \alpha \le J$. We may assume that 
		\begin{equation*}
		\begin{aligned}
			\mathcal{L}(\alpha)\le \mathcal{L}(1)   < E_{max} - \eta,\\
			L(g) < E_{max} - \eta,
		\end{aligned}
		\end{equation*}
		for some $\eta > 0$. Let $G$ and $U_k$ denote the maximal strong solution to \eqref{cubic-intro} with initial data $g$ and $\tilde{\psi}^{\alpha}_{\mathcal{O}^{\alpha}_k}$ respectively. From the definition of $E_{max}$, we can find that $G$ and $U_k(0)$ are global. Furthermore, passing to a subsequence, we have
		\begin{equation}\label{fml-thm-Uk-bound}
			\|G\|_{Z([-1,1])} + \|U_k\|_{Z([-1,1])} \lesssim 1.
		\end{equation}
		We write $U^{\alpha}_0 := G$ and define
		\begin{equation*}
			U^J_{prof, k} := G + \sum_{\alpha = 1}^J U^{\alpha}_k 
		\end{equation*}
		for every $J,k \in \Z$.
		
		We prove for $k$ large enough, 
		\begin{equation}\label{fml-thm-Uprofk-bound}
			\|U^J_{prof, k}\|_{X^{\frac{1}{2}}(-1,1)} \lesssim 1,
		\end{equation}
		uniformly in $J$. For any $0 < \delta < 1$, there are at most $M := \delta^{-1}E_{max}$ number of $\alpha$ such that $\mathcal{L}(\alpha) > \delta$. Hence, for all $\alpha > M$, we have $\mathcal{L}(\alpha) \le \delta$. Consequently, we obtain $\|U^{\alpha}_k(0)\|_{H^{\frac{1}{2}}(\T^3)} \le \mathcal{L}(\alpha)^{\frac{1}{2}} \le \delta^{\frac{1}{2}}$ for $\alpha > M$ and sufficiently large $k$. By the small data global well-posedness, we get a $U_k^{\alpha} \in X^{\frac{1}{2}}(-1,1)$ solve \eqref{cubic-intro} with initial data $U_k^{\alpha}(0)$, since $\delta$ can be taken arbitrarily. Finally, we get
		\begin{align*}
			\|U_k^{\alpha}\|_{X^{\frac{1}{2}}(-1,1)} \le & \|G\|_{X^{\frac{1}{2}}(-1,1)} + \sum_{\alpha = 1}^{M}\|U^{\alpha}_k\|_{X^{\frac{1}{2}}(-1,1)}\\ 
			& + \sum_{\alpha = M + 1}^J\|U^{\alpha}_k(t) - e^{it\Delta}U^{\alpha}_k(0)\|_{X^{\frac{1}{2}}(-1,1)} + \sum_{\alpha = M + 1}^J\|e^{it\Delta}U^{\alpha}_k(0)\|_{X^{\frac{1}{2}}(-1,1)}\\
			\lesssim & 1 + M + E_{max} + \bigg\|\sum_{\alpha = M+1}^J U_k^{\alpha}(0) \bigg\|_{H^{\frac{1}{2}}(\T^3)}.
		\end{align*}
		Thanks to almost orthogonality properties in Lemma \ref{lem-ep-decomp}, we have
		\begin{equation*}
			\bigg\| \sum_{\alpha = M+1}^J U_k^{\alpha}(0) \bigg\|^2_{H^{\frac{1}{2}}(\T^3)} = \sum_{\alpha = M+1}^J \|U_k^{\alpha}(0)\|^2_{H^{\frac{1}{2}}(\T^3)} + o_k(1) \lesssim E_{max}.
		\end{equation*}
		The proof of \ref{fml-thm-Uprofk-bound} has accomplished.
		
		Then, we define
		\begin{equation*}
			U_{app,k}^J := U_{prof,k}^J + e^{it\Delta}R^J_k,
		\end{equation*}
		to be an approximating solution to \eqref{cubic-intro} with error term
		\begin{equation}\label{fml-thm-error-term}
			e = \sum_{\alpha = 0}^J F(U^{\alpha}_k) - F\bigg(\sum_{\alpha = 0}^J U^{\alpha}_k + e^{it\Delta}R^J_k\bigg),
		\end{equation}
		where $F(u) = |u|^2u$. We need the following lemma to control the error term in \eqref{fml-thm-error-term}.
		\begin{lemma}\label{lem-thm-error-bound}
			Let $e$ be the same as in \eqref{fml-thm-error-term}, then
			\begin{equation*}
				\limsup_{J \to \infty}\limsup_{k \to \infty} \|e\|_{N(I_k)} = 0.
			\end{equation*}
		\end{lemma}
		Lemma \ref{lem-thm-error-bound} can be obtained by Proposition \ref{prop-ep-nonlin-ortho}, the proof is similar with \cite{Yu2021} and we omit it. Assuming Lemma \ref{lem-thm-error-bound}, for any $\varepsilon > 0$ we can find $J \ge J(\varepsilon)$ such that
		\begin{equation}
			\limsup_{k \to \infty} \|e\|_{N(I_k)} \le \frac{\varepsilon}{2}.
		\end{equation}
		By \eqref{fml-thm-Uprofk-bound}, we can find $\|U_{app,k}^J\|_{X^{\frac{1}{2}}(I_k)}$ is uniformly bounded in $J$. Apply this with Proposition \ref{prop-stability}, we have
		\begin{equation*}
			\|u_k\|_{X^{\frac{1}{2}}(I_k)} \lesssim \|U^J_{app,k}\|_{X^{\frac{1}{2}}(I_k)} \le \|U^J_{prof,k}\|_{X^{\frac{1}{2}}(-1,1)} + \|e^{it\Delta}R^J_k\|_{X^{\frac{1}{2}}(-1,1)} \lesssim 1,
		\end{equation*}
		for $k$ large enough. We also obtain a contradiction with \eqref{fml-thm-Emax}. So far, we have finished the proof of Theorem \eqref{thm-finite-energy}.
	\end{proof}

\vspace{5mm}

\end{document}